\documentclass[a4paper]{article}

\usepackage[utf8]{inputenc}
\usepackage{amsthm,amsmath,amssymb,amsfonts}

\usepackage{url}
\usepackage{esvect}

\usepackage[scale={.8,.75}]{geometry}

\usepackage{graphicx}
\usepackage{dsfont}
\usepackage[usenames,dvipsnames,svgnames,table]{xcolor}
\usepackage[colorlinks=true, pdfstartview=FitV,linkcolor=ForestGreen,citecolor=ForestGreen, urlcolor=blue]{hyperref}

\newcommand{\dx}{\mathrm{d}x}
\newcommand{\dy}{\mathrm{d}y}
\newcommand{\dt}{\mathrm{d}t}
\newcommand{\R}{\mathbb{R}}
\newcommand{\RN}{\mathbb{R}^N}
\DeclareMathOperator{\supp}{supp}

\newcommand{\eps}{\epsilon}
\newcommand{\un}{\mathbb{I}}

\newcommand{\DG}{\mathcal{D}_G}
\newcommand{\BG}{\mathcal{B}_G}
\newcommand{\QG}{\mathcal{Q}_G}
\newcommand{\Epm}{\mathcal{E}_\pm}
\newcommand{\Ep}{\mathcal{E}_+}
\newcommand{\get}{\mathfrak{G}}
\newcommand{\fipm}{\varphi_\pm}
\newcommand{\fip}{\varphi_+}
\newcommand{\fim}{\varphi_-}
\newcommand{\Qint}{Q_{\mathrm{int}}}
\newcommand{\Qout}{Q_{\mathrm{out}}}
\newcommand{\QGint}{Q_G^{\mathrm{int}}}
\newcommand{\QGout}{Q_G^{\mathrm{out}}}
\newcommand{\ufi}{u_\varphi^\pm}
\newcommand{\ufip}{u_\varphi^+}
\newcommand{\ufim}{u_\varphi^-}
\newcommand{\bufim}{\bar{u}_\varphi^-}
\newcommand{\fik}{\varphi_k}

\newcommand{\fikmun}{\varphi_{k-1}}

\newtheorem{definition}{Definition}[section]
\newtheorem{theorem}{Theorem}[section]
\newtheorem{proposition}{Proposition}[section]

\newtheorem{lemma}{Lemma}[section]

\theoremstyle{remark}
\newtheorem{remark}{Remark}

\title{Regularity of solutions of a fractional porous medium equation}
\author{ Cyril Imbert\footnote{CNRS, \'Ecole normale supérieure, 45
    rue d'Ulm 75005 Paris, France, Mail: Cyril.Imbert@ens.fr..}\:,
  Rana Tarhini\footnote{Universit\'e Paris Est, Laboratoire d'Analyse
    et de Math\'ematiques Appliqu\'ees UMR 8050, 61 avenue du
    G\'en\'eral de Gaulle, 94010 Cr\'eteil, France, Mail:
    rana.tarhini@math.cnrs.fr.}\:,
  François Vigneron\footnote{Universit\'e Paris Est, Laboratoire
    d'Analyse et de Math\'ematiques Appliqu\'ees UMR 8050, 61 avenue
    du G\'en\'eral de Gaulle, 94010 Cr\'eteil, France, Mail:
    francois.vigneron@u-pec.fr.}  }
\begin{document}
\maketitle

\begin{abstract}
  This article is concerned with a porous medium equation whose pressure
  law is both nonlinear and nonlocal, namely
  \[
  \partial_t u = {  \nabla \cdot} \left(u \nabla(-\Delta)^{\frac{\alpha}{2}-1}u^{m-1} \right)
  \] 
  where $u:\mathbb{R}_+\times \mathbb{R}^N \to \mathbb{R}_+$, for $0<\alpha<2$ and $m\geq2$.
  We prove that the $L^1\cap L^\infty$ weak solutions constructed by Biler, Imbert and Karch (2015) 
  are locally Hölder-continuous in time and space.
  In this article, the classical parabolic De Giorgi techniques for the regularity of PDEs are tailored to fit this particular variant of the PME equation.
  In  the spirit of the work of Caffarelli, Chan and Vasseur (2011), 
  the two main ingredients are the derivation of local energy estimates
  and a so-called  ``intermediate value lemma''.
  For $\alpha\leq1$, we adapt the proof of Caffarelli,
  Soria and V\'azquez~(2013), 
  who treated the case of a linear pressure law.
  We then use a non-linear drift to cancel out the singular terms that would otherwise appear in the energy estimates.
  
\bigskip\noindent%
\textbf{MSC Primary:} 35B65.\\
\textbf{MSC Secondary:} 76S05, 35K55, 45P05, 45K05, 47G10.
\end{abstract}

\section{Introduction}

In this work, we study the regularity of non-negative weak solutions of the following degenerate
nonlinear nonlocal evolution equation
\begin{equation} \label{e:main}
\partial_t u = {  \nabla \cdot} \left(u \nabla^{\alpha-1}G(u) \right), \quad t >0, \quad  x \in \mathbb{R}^N, 
\end{equation}
where $G(u) = u^{m-1}$ with $m \ge 2$.
The equation is supplemented with initial data
\begin{align} \label{e:ic}
u(0,x) = u_0(x),
\end{align}
which we will assume to be both non-negative and integrable on $\RN$.

 \bigskip For $\alpha \in(0,2)$, the symbol $\nabla^{\alpha-1}$
 denotes the integro-differential operator
 $\nabla(-\Delta)^{\frac{\alpha}{2}-1}$.  It is a nonlocal operator of
 order $\alpha-1$. For a smooth and bounded function $v$, it has the
 following singular integral representation
 \begin{equation}\label{fractional}
\nabla^{\alpha-1} v(x)= c_{ \alpha,N} \int_{ \RN} \left(v(y)-v(x) \right)\frac{y-x}{|y-x|^{N+\alpha}} dy
\end{equation} 
with a suitable constant $c_{\alpha,N}$. Moreover, we have
${\nabla\cdot} \nabla^{\alpha-1} = -(-\Delta)^\frac{\alpha}{2}.$

Our main result is the H\"older regularity of weak solutions of
\eqref{e:main}. For short, let us write $Q_T = (0,T)\times\RN$.
\begin{definition}[Weak solutions]
  A function $u : Q_T \rightarrow \mathbb{R}$ is a weak solution of
  \eqref{e:main}-\eqref{e:ic} if $u \in L^1(Q_T)$,
  $\nabla^{\alpha-1}(|u|^{m-2}u)\in L^1_{\operatorname{loc}}(Q_T)$,
  $|u|\nabla^{\alpha-1}(|u|^{m-2}u)\in L^1_{\operatorname{loc}}(Q_T)$
  and
\begin{equation}\label{weak}
\iint u \partial_t\varphi \dt\dx - \iint |u|\nabla^{\alpha-1}(|u|^{m-2}u).\nabla\varphi \dt\dx = -\int u_0(x)\varphi(0,x)\dx
\end{equation}
for all test functions $\varphi \in C^\infty(Q_T)\cap C^1(\bar{Q}_T)$ { with} compact support in the space
variable $x$ and { that} vanish near $t = T$.
\end{definition}
\begin{theorem}[Hölder regularity] \label{thm:main} Let us assume that
  $\alpha \in (0,2)$ and $m \ge 2$.  For any initial data
 \[ u_0\in L^1\cap L^\infty(\RN;\mathbb{R}_+),\] weak solutions $u$
  of~\eqref{e:main}-\eqref{e:ic} are H\"older continuous at strictly
  positive times. More precisely, there is  $\beta \in (0,1)$
  depending only on $N$, $m$ and $\alpha$ such that for all $T_0,T_1 >0$ with $T_0 < T_1$, 
\begin{equation}
[u]_{C^{\beta}([T_0,T_1]\times\RN)} \leq C \|u\|_{L^\infty([T_0/2,T_1]\times\RN)}
\end{equation}
where $C$ only  depends on $N$, $m$ and $\alpha$ and $T_0$.
\end{theorem}
\begin{remark}
  Weak solutions have been constructed in \cite{BIK} under the assumptions of theorem~\ref{thm:main} and even
  for a range of values of $m<2$ as well.
  For a precise statement, see theorem~\ref{existence} below.
  Our proof might be adapted to those small values of $m$, but it  will require some modifications and additional work.
  The key change would be the loss of convexity of $G$, which would immediately void~\eqref{e:DG}, \eqref{e:DGconvex} and \eqref{e:DG2}.
\end{remark}
\begin{remark}
  The (linear) case $m=2$ was treated in \cite{CSV,CV}  for any $\alpha\in(0,2)$.
\end{remark}
\begin{remark}
  One could also study the regularity of unsigned weak solutions of the equation
  \eqref{BIK} below, which is the unsigned version of the equation~\eqref{e:main}.
  Since the subsequent proof is local, it is probably possible to extend our result in this direction.
\end{remark}

\paragraph{Review of the literature.}

Let us briefly recall how the porous medium equation is derived from
the law of conservation of mass, for a gas propagating in a
homogeneous porous medium \cite{aronson, Vaz}:
\begin{equation*}
\partial_t u + {\nabla \cdot} (uv) = 0.
\end{equation*} 
In this equation, $u \geq 0$ denotes the density of the gas and
${ v\in\RN}$ is the locally averaged velocity.  Darcy's law states
that $v = -\nabla p$, where $p$ denotes the pressure. Finally, the
pressure law implies that $p$ is a monotone operator of $u$ \textit{i.e.}
$p = f(u)$. This leads us to the following equation
\begin{equation}
\partial_t u = {\nabla \cdot} (u\nabla f(u)).
\end{equation}
The case $p = u$ is the simplest  pressure law and leads to the
  Boussinesq's equation \cite{Bear, Boussinesq}:
\begin{equation}
\partial_t u = c \Delta(u^2).
\end{equation}

\medskip  L.~Caffarelli and J. L.~V\'azquez \cite{CafVaz} studied the following equation:
\begin{eqnarray} \label{CSV}
\partial_t u =
{ \nabla \cdot} \left(u \nabla (-\Delta)^{-s}u \right), \quad t > 0, \quad x \in \mathbb{R}^N.
\end{eqnarray}
This equation was proposed by \cite{CafVaz} to add long-distance effects in the physical model
(for further details, see the motivations therein).  They study this problem
with non-negative initial data that are integrable and
decay at infinity. For $s=\frac{2-\alpha}{2}\in(0,1)$ and $m = 2$, our
equation \eqref{e:main} coincides with \eqref{CSV}.

The existence of mass-preserving non-negative weak solutions satisfying
energy estimates has been proved in \cite{CafVaz}. Such solutions
  have a finite propagation speed.  Their asymptotic behavior as
$t\rightarrow \infty$ has been studied in \cite{CafVaz1}. Moreover, in
\cite{CSV} and \cite{CV}, the boundedness and the Hölder regularity
of non-negative solutions have been obtained for $s \in (0,1).$

 The proof of the Hölder regularity in the range
$s \in (0,1/2)$ is based on De Giorgi-type oscillation lemmas and 
  on the scaling property (see \eqref{scaling} below) of the equation.  For a
  general review of the De Giorgi method for classical elliptic and
parabolic equations, { we refer for instance to}~\cite{DeGiorgi},
\cite{Lieberman}, \cite{Vasseur} and~\cite{CafVas}.  The regularity
result in the case $s \in (1/2,1)$, which corresponds to
$\alpha \in (0,1)$ { for us}, is more difficult due to convection
effects that appear and make some integrals diverge. The method
proposed in \cite{CSV} { consists in a} geometrical transformation 
  that absorbs the uncontrolled growth of one of the integrals that
appear in the iterated energy estimates.

The most delicate situation, which is the case $s = 1/2$, has been
treated in \cite{CV}. The authors performed an iteration analysis that
combines consecutive applications of scaling and geometrical
transformations.

 A similar De Giorgi method is also used in \cite{CCV} to prove 
  the Hölder regularity for nonlinear nonlocal time-dependent
variational equations. In this case however, $-u$ satisfies the
same equation as $u$, which slightly simplifies the proof.

\bigskip In \cite{BIK1} and \cite{BIK}, P.~Biler, C.~Imbert and
G.~Karch consider a problem similar to \eqref{e:main}-\eqref{e:ic},
but with $u$ unsigned. They prove, under some conditions on $m$ (see
\eqref{cond:m} below), the existence of bounded and mass-preserving
weak solutions for the Cauchy problem
\begin{equation} \label{BIK}
\partial_t u =
{\nabla \cdot} \left(|u| \nabla^{\alpha-1}(|u|^{m-2}u) \right), \quad t > 0, \quad x \in \mathbb{R}^N,
\end{equation}
with initial condition 
\begin{equation} \label{init}
u(0,x) = u_0(x)
\end{equation}
where $u_0$ is an integrable but not necessarily positive function on
$\mathbb{R}^N$. Moreover, they show that the solution $u$ is
non-negative if the initial condition $u_0$ is too, in which case the
solution is a solution of our problem \eqref{e:main}-\eqref{e:ic}.
In the sequel of this paper, this existence result is our starting point.
The finite speed of propagation for these non-negative weak solutions has
been proved in \cite{I} and holds under the same conditions on $m$.

\bigskip A variant of the porous medium equation with both a
fractional potential pressure and a fractional time derivative has
been studied in \cite{FPMET}:
\begin{eqnarray*}
D^\alpha_t u - { \nabla\cdot}(u\nabla(-\Delta)^{-\sigma}u) = f,
\end{eqnarray*}
where $D^\alpha_t$ is a Caputo-type time derivative. The authors study
both the existence and  the H\"older regularity of the
solutions,  using the De Giorgi method as in \cite{CSV}.

\paragraph{Organization of the paper and general ideas.}

This paper is organized as follows:
\begin{itemize}
\item In Section 2 we recall briefly how the existence theorem~\ref{existence} was
    established.
\item Section 3 is devoted to the local energy estimates satisfied by
  a bounded weak solution. We first derive general energy estimates
  (theorem~\ref{thm:EE}) and then we localize them and we improve them
  by estimating in a more precise way the ``dissipation'' terms
  (proposition~\ref{prop:EE-local}). In this section, we separated the 
  arguments for $\alpha\in(1,2)$ from the ones for $\alpha\in(0,1]$.
\item In Section 4 we prove the first lemmas of De Giorgi.  The idea
    is that a direct application of the energy estimate along a
    sequence of macroscopic space-time balls leads to a point-wise
    upper-bound, provided that the measure of the set where $u$ is
    small is \textsl{sufficiently large}. Similarly, one can get a
    point-wise lower-bound from knowing that $u$ is large enough on a
    large set.

  In Section 5, we move on to the lemma on intermediate values. 
    It roughly claims that if both the sets where $u$ is small and
    where $u$ is large are substantial measure-wise, then, thanks to
    the ``good extra term'' of the local energy balance, $u$ also has
    to spend a substantial space-time in between. In naive words, we
    quantify the cost of oscillations.

    In Section 6, this idea allows us to subtly improve the first
    lemma of De Giorgi: the point-wise upper bound can be ascertained
    provided only that the measure of the set where $u$ is small is
    \textsl{not too small}. The proof comes naturally by
    contradiction: if the upper bound could not be improved, then too
    much energy would be lost in the oscillations induced between the
    maximal point and the low values set. Section 6 seems to be a
    subtle refinement of Section 4, but it suffices to prove
    theorem~\ref{thm:main}.
\item In Section~7,  one follows a ``zoom-in and enlarge'' sequence
    of solutions, along which the oscillation is controlled either
    from above by the refined first De Giorgi lemma of Section 6 or from
    below by the crude one of Section 4.  The improvement of Section~6
    was needed to have a clean alternative at this point.  This scheme leads
    directly to the Hölder regularity of the solution $u$.
\end{itemize}

\section{Preliminaries}\label{preliminaries}

\paragraph{Notations.}  

In this work, we denote by $B_r$ the ball of $\RN$ of radius $r > 0$
and of center 0.
For any measurable function $v$ we define its positive and negative part by:
\begin{equation}\label{vplusmoins}
v_+ = \max(0,v) \quad\text{and}\quad v_- = \max(0,-v).
\end{equation}
We will often use the following notation and identities:
\[a\vee b=\max\{a,b\}=a+(b-a)_+ \quad\text{and}\quad a\wedge b = \min\{a,b\}=a-(b-a)_-.\]

The fractional Laplacian has the following singular integral expression:
\begin{equation}\label{fractional_laplace}
(-\Delta)^{\alpha/2}v(x) = - \int_{\RN} (v(y)-v(x)) \frac{c^0_{\alpha,N}}{|x-y|^{N+\alpha}}
dy,
\end{equation}
where $c^0_{\alpha,N}$ is a constant only depending on $\alpha$ and $N$. 

 Finally, let us point out that we will usually specify the domain of each integral,
except for double space integrals, where $\iint f(x,y)\dx\dy$
will denote an integral over $\RN\times\RN$, unless stated otherwise.

\paragraph{Weak   solutions.}

The existence of positive weak solutions for our Cauchy problem at
hand~\eqref{e:main}-\eqref{e:ic} was proved in \cite{BIK}. 
\begin{theorem}[{Existence of weak solutions, from \cite[{theorem~2.6}]{BIK}}]\label{existence}
Let $\alpha \in (0,2)$ and
\begin{equation}\label{cond:m}
m > \max\left\{1 + \frac{1-\alpha}{N}  ; 3-\frac{2}{\alpha}\right\}.
\end{equation}
For any $u_0 \in L^1(\RN;\mathbb{R}_+)$, the Cauchy problem \eqref{e:main}-\eqref{e:ic} admits
a weak solution $u$ on $(0,+\infty)\times\RN$.
Moreover,
\begin{equation}
\int_{\RN} u(t,x) \dx = \int_{\RN} u_0(x) \dx
\end{equation}
and for each $p \in [1,\infty]$ and $t>0$
\begin{align} \label{boundedness}
\Vert u(t)\Vert_p \leq \min\left\{C_{N,\alpha,m} \Vert u_0\Vert_1^\frac{N(m-1)/p+\alpha}{N(m-1)+\alpha} t^{-\frac{N}{N(m-1)+\alpha}\left(1-\frac{1}{p}\right)} ; \Vert u_0\Vert_p \right\}.
\end{align}
The constant $C_{N,\alpha,m}$ is independent of $p$, $t$ and $u_0$.
\end{theorem}

The admissible pairs of $(\alpha,m)$ in theorem~\ref{existence} are illustrated on the following drawing. 
\begin{center}
\includegraphics[width=200pt]{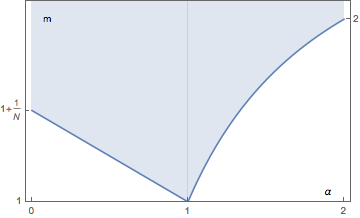}
\end{center}
However, in the rest of this paper, we restrict ourselves to the case $m\geq2$, even though $\alpha$ spans the whole range $(0,2)$.

\paragraph{Scaling invariance of the equation.}

The solutions of~\eqref{e:main} have the following scaling property.
\begin{lemma} \label{l:scaling}
If $u$ satisfies \eqref{e:main} then
\begin{equation}\label{scaling}
u_{A,B,C}(t,x) = Au(Bt,Cx)
\end{equation}
also satisfies \eqref{e:main}, provided that $B = A^{m-1} C^\alpha$.
\end{lemma}

\begin{remark}
In~\eqref{scaling}, both parameters $A$ and $C$ can take arbitrary values.
It is therefore possible to rescale the physical space independently
from a change of amplitude of the solution. This double-scaling property plays a key role
in the final argument of the proof of theorem~\ref{thm:main} (see \S\ref{mainproof}).
\end{remark}

\paragraph{A characterization of the Hölder continuity.}

To prove the Hölder regularity we will use the following lemma, which
is part of the folklore:
\begin{lemma} \label{l:holder} Let $u$ be a function defined in
  ${ (-1,0)}\times B_1$ such that for any
  $(t_0,x_0) \in (-1/2,0)\times B_{1/2}$ and any $r \in (0,1/2)$ we
  have
\[
\underset{(t_0-r,t_0)\times B_r(x_0)}{\operatorname{osc}u} \leq C r^\beta.
\]
Then $u$ is $\beta$-Hölder continuous in { $(-1/2,0)\times B_{1/2}$}.
\end{lemma}

\paragraph{Sobolev embedding.}

The following local Sobolev embedding theorem will be useful:
\[
H^\frac{\alpha}{2}(B_r) \subset L^p(B_r)
\]
for $p={ \frac{2N}{N-\alpha}} > 2$ and any $r >0$. More precisely,
there is a constant $C$, independent of $r$, such that:
\begin{equation}\label{sobolev}
\left(\int_{B_r} u^p\dx\right)^\frac{2}{p} \leq C \iint_{ B_r \times
  B_r} \left(u(y)-u(x)\right)^2 \frac{\dx\dy}{|y-x|^{N+\alpha}}.
\end{equation}

\section{Energy estimates}

In this section, we derive the necessary energy estimates to follow De
Giorgi's original path towards the H\"older continuity of the solutions.
As we will ultimately use lemma~\ref{l:holder} on a dyadic rescaled sequence of solutions,
we cannot take for granted the value of the $L^\infty$ bound of the weak solution.
Instead, we have to prove the energy estimates for weak solutions that are potentially
allowed to grow as a mild power-law at infinity.

\begin{definition}
For any $\epsilon>0$, let us define
\begin{equation}\label{defpsiepsilon}
\Psi_\epsilon(x) = (|x|^\epsilon-2)_+.
\end{equation}
\end{definition}

\begin{theorem}[Energy estimates] \label{thm:EE} Let us assume that
  $\alpha\in (0,2)$ and that $m \ge 2$. Then there are absolute constants $\eps_0>$
  and $C>0$ (depending only on $N,\alpha,m$) such that, for any
 weak solution $u$ of~\eqref{e:main} in $(-2,0] \times \RN$
  satisfying   for some $\eps \in (0,\eps_0)$:
\begin{equation}\label{psiepsilon}
  \forall t\in(-2,0],\quad \forall x\in\RN, \qquad 0\leq u (t,x) \leq 1 + \Psi_\epsilon(x)
\end{equation}
and for any
  smooth truncation functions $\varphi_\pm:\RN \to [0,+\infty)$ such that:
\begin{itemize}
\item $1/4<\varphi_+ \leq 1+ \Psi_\eps$ on $\RN$ with
 $\varphi_+ = 1+ \Psi_\eps$ outside $B_{2^{1/\eps}}$ and
\[
|\nabla \fip / \fip| 
+ |\nabla \fip| + |\nabla \fip / \fip|^2 + |\nabla \fip|^2 \le C_{\fip},\]
\item  $0<\varphi_-\leq 1$ on $B_2$ but $\varphi_- \equiv 0$ outside $B_2$  with
\[|\nabla \fim /\fim | \le C_{\fim} \fim^{-1/m_0}\]
on $\bar{B}_2$ for some $m_0 \ge 2$,
\end{itemize}
the two following energy estimates hold true for any  $-2<t_1<t_2<0$:
\begin{align}\label{e:EE}
\nonumber \frac12\int_{\RN} & \left(u(t_2,x)-\fipm (x) \right)_\pm^2\, \fipm^{-1} (x) \dx  \\
\nonumber  &
+ \frac14 \int_{t_1}^{t_2} \iint \bigg((u(t,y) - \fipm(y))_\pm -(u(t,x) - \fipm(x))_\pm\bigg)^2 
\DG (u(t,x),u(t,y)) \frac{\dx \dy }{|y-x|^{N+\alpha}} \dt   \\ 
  &+ \frac14 \int_{t_1}^{t_2} \iint (u(t,x) - \fipm(x))_+ (u(t,y) - \fipm(y))_- 
\DG (u(t,x),u(t,y)) \frac{\dx \dy }{|y-x|^{N+\alpha}}  \dt  \\
\nonumber & \leq \int_{\RN}  \left(u(t_1,x)-\fipm(x)\right)_\pm^2 \fipm^{-1}(x) \dx  
+C C_{\fipm} \,  \left|\{(u - \fipm)_\pm > 0 \}\cap (t_1,t_2)\times\RN \right|,
\end{align}
where $\DG$ is defined for $a,b \in \R$ by \( \DG (a,b) = \frac{G(a)-G(b)}{a-b}\).
When $\alpha\in(0,1]$, the estimates \eqref{e:EE} hold for a drifted solution $\bar{u}$ 
defined by~\eqref{uDrift1}-\eqref{uDrift2}.
\end{theorem}
\begin{remark}
The $\pm$ notation means that the inequality~\eqref{e:EE} stands true if all the
symbols $\pm$ are either simultaneously replaced by $+$ or by $-$.
Hybrid choices are not allowed.
\end{remark}
\begin{remark}
The functions $\varphi_\pm$ serve a truncation purpose, which should become clear
as the proof unfolds. For example, $(u-\varphi_+)_+\equiv0$ outside $B_{2^{1/\eps}}$
and similarly $(u-\varphi_-)_-\equiv0$ outside $B_2$, which in particular
takes the ambiguity out of the first integral as $\varphi_-^{-1}$ does not have to be computed
ouside $B_2$.
\end{remark}
\begin{remark}
Obviously, each term of \eqref{e:EE} is non-negative.
 The third term in \eqref{e:EE} that mixes a positive and a negative part
 is called the ``good extra term'' in \cite{CCV}.  It will play a
  crucial role in the proof of the lemma on intermediate values (see
  Section 5).  By themselves, the other non-negative terms of \eqref{e:EE} would be
  sufficient to prove the first lemmas of De Giorgi (see Section 4).
\end{remark}
\begin{remark}
In order to prove the energy estimates we will introduce an alternate energy functional:
\begin{equation} \label{e:alternate}
\Epm (t) = \int_{\RN} H \left(1 \pm
  \frac{(u-\varphi_\pm)_\pm(x)}{\varphi(x)_\pm}\right) \varphi_\pm (x) \dx 
  \end{equation}
where $H$ is an appropriate convex function. The functional $\mathcal{E}_+$ is
well-defined since $\varphi_+$ does not vanish.
Note also that $\frac{(u-\varphi_+)_+}{\varphi_+} \in [0,(\inf_{B_{2^{1/\epsilon}}} \varphi_+)^{-1}]\subset[0,4]$.
As far as $\mathcal{E}_-$ is concerned,
we remark that 
\( 1 - \frac{(u-\varphi_-)_-}{\varphi_-} = 1 \wedge \frac{u}{\varphi_-}\in[0,1] \).
In particular, the spurious fraction simply boils down to $H(1)$ when $x\not\in B_2$.
Moreover, only the values of $H(1+r)$ for $r\in[-1,4]$ are relevant for~\eqref{e:alternate}.
\end{remark}

\bigskip
The proof of theorem~\ref{thm:EE} is structured as follows.
First  we explain why it is enough to consider the alternate energy
functional~\eqref{e:alternate}.
Then we estimate the error terms for the energy $\int (u -\fip)_+^2 \fip^{-1} $. 
Next, we deal with the case of $\int (u-\fim)_-^2 \fim^{-1}$.
Finally, we explain the modifications that are necessary to deal with the case $\alpha\in(0,1]$.

\subsection{An alternate energy functional}

We consider the convex function $H: [0,+\infty) \to [0,+\infty)$ such that
\[ H'' (r) = r^{-1} \quad \text{and} \quad H (1)=H'(1) = 0.\]
The function $H$ is given by the formula $H(r) = r \ln r -r +1$.  
Following~\cite{CSV}, we consider the energy functional~\eqref{e:alternate}.
As $ \frac14 r^2 \le H (1+r) \le r^2$ for $r \in [-1,4]$, the proof of
theorem~\ref{thm:EE} is reduced to proving that
\begin{multline} \label{eq:goal-ee}
  \Epm (t_2) + \int_{t_1}^{t_2} \left( \BG((u(t) -\fipm)_\pm,(u(t) -\fipm)_\pm) - \BG((u(t) -\fipm)_+,(u(t) -\fipm)_-)\right) \dt \\
  \le \Epm (t_1) +C \left|\{(u - \fipm)_\pm > 0 \}\cap    (t_1,t_2)\times\RN \right|
\end{multline}
where the bilinear form $\BG$ is defined as follows:
\begin{equation} \BG (v,w) = \iint (v(y)-v(x))(w(y)-w(x)) \DG (u(x),u(y)) \frac{c_{\alpha,N}^0 \dx \dy }{|y-x|^{N+\alpha}}\cdotp\end{equation}
Let us recall that \(\DG (a,b) = \frac{G(a)-G(b)}{a-b}.\)

\bigskip
Le us compute first the time derivative of the alternate energy functional:
\label{IPP}
\begin{align*} 
\frac{d}{dt} \Epm (t) &= \int_{(u-\fipm)_\pm>0}  H' \left(1 \pm \frac{(u-\fipm)_\pm}{\varphi_\pm} \right)  \partial_t u \dx\\
&=\int  H' \left(1 \pm \frac{(u-\fipm)_\pm}{\varphi_\pm} \right)  \partial_t u \dx\\
& = - \int H'' \left(1 \pm \frac{(u-\fipm)_\pm}{\varphi_\pm} \right) u \nabla^{\alpha-1} G(u) \cdot \nabla \left(\frac{\pm(u-\fipm)_\pm}{\fipm} \right). 
\end{align*}
This formal computation can be made rigorous thanks to the regularity of some approximate
solutions as it was done in \cite{BIK}. 
We now remark that, on the set $\{\pm(u-\fipm)>0\}$, we have the following remarkable identity:
\[ H'' \left(1 \pm \frac{(u-\fipm)_\pm}{\varphi_\pm} \right) u = \fipm .\]
This implies that 
\begin{align}
\nonumber 
\frac{d}{dt} \Epm (t) = &- \int \nabla^{\alpha-1} G(u) \cdot \nabla (\pm (u-\fipm)_\pm) \dx  
+ \int (\pm (u-\fipm)_\pm) \nabla^{\alpha-1} G(u) \cdot \frac{\nabla \fipm}{\fipm} \dx \\
= & - \BG (\pm (u-\fipm)_\pm, u ) + \QG (\pm (u-\fipm)_\pm, u)
\label{e:E'interm}
\end{align}
with 
\begin{equation}
\QG (v,w) = \iint v(x) (w(y)-w(x)) \DG (u(x),u(y)) \frac{\nabla \fipm (x)}{\fipm(x)} \cdot (y-x) \frac{c_{\alpha,N} \dx \dy }{ |y-x|^{N+\alpha}}\cdotp
\end{equation}

\bigskip
Up to now, the second variable $w=u$ of $\BG(\cdot,u)$ or $\QG(\cdot,u)$
could have been simplified with the denominator of the kernel $\DG (u(x),u(y))$. We are now
going to split that second variable.
More precisely, using the fact that $u = \pm (u- \fipm)_\pm \mp (u-\fipm)_\mp + \fipm$, we get:
\[ \BG (\pm (u-\fipm)_\pm, u ) = \BG ( (u-\fipm)_\pm, (u-\fipm)_\pm ) - \BG ( (u-\fipm)_+, (u-\fipm)_-) + \BG (\pm (u-\fipm)_\pm,\fipm).\]
Combining this with \eqref{e:E'interm} yields
\begin{align} \label{e:E'start}
\frac{d}{dt} \Epm (t) + \BG ( (u-\fipm)_\pm, (u-\fipm)_\pm ) &- \BG ( (u-\fipm)_+, (u-\fipm)_-) \\
&= - \BG (\pm (u-\fipm)_\pm, \fipm ) + \QG (\pm (u-\fipm)_\pm, u). \notag
\end{align}
We remark that the two terms $\BG ( (u-\fipm)_\pm, (u-\fipm)_\pm )$
and $- \BG ( (u-\fipm)_+, (u-\fipm)_-)$ are non-negative. Following \cite{CCV,CSV}, the first
one is referred to as the ``coercive term'' while the second one is
referred to as the ``good extra term''. The rest
of the proof of theorem~\ref{thm:EE} consists in controlling the terms in the right hand side
of \eqref{e:E'start} by those two non-negative terms plus
$C |\{ (u-\fipm)_\pm >0 \} \cap (t_1,t_2) \times \R^N|$.

\begin{proof}[Proof of theorem~\ref{thm:EE}]
  First, let us consider the case $\alpha\in(1,2)$.
  As far as $\mathcal{E}_+$ is concerned, one has to combine~\eqref{e:E'start}
  with the subsequent lemmas~\ref{l:ee1}, \ref{l:ee2+}, \ref{l:ee3+}, \ref{l:ee4+},
  \ref{l:ee5+}, \ref{l:ee6+} and \ref{l:ee7+}.  As far as
  $\mathcal{E}_-$ is concerned, one has to combine \eqref{e:E'start} with the subsequent
  lemmas~\ref{l:ee1}, \ref{l:ee2-} and \ref{l:ee3-}. 
  Given the range of admissible parameters $(\alpha,m)$ in theorem~\ref{thm:EE},
  the critical value for $\epsilon_0$  when $\alpha\in(1,2)$ is \[\epsilon_0=(\alpha-1)/(m-1).\]
  When $\alpha\in(0,1]$, the energy estimate~\eqref{e:E'start} is replaced by~\eqref{e:E'startNEW}.
  For the energy estimate on $\mathcal{E}_+$, lemmas \ref{l:ee5+}, \ref{l:ee6+}, \ref{l:ee7+} are replaced by lemma \ref{l:ee567+}.
  For the energy estimate on $\mathcal{E}_-$, lemmas~\ref{l:ee2-}, \ref{l:ee3-} are updated by lemma \ref{l:ee23-}.
 In that case, the critical value for $\epsilon_0$ is
  \[
  \epsilon_0 = \alpha/m.
  \]
  The proof will be complete once the lemmas are established.
\end{proof}

\bigskip
In what follows it will be convenient to write
\[ \ufi = (u-\fipm)_\pm.\]
An inequality $A\leq cB$ that involves a universal constant $c$ depending on $N,\alpha,m$ and $C_{\fipm}$ will be denoted by~$A\lesssim B$.

\medskip
 We will repeatedly use the fact that \eqref{psiepsilon} implies
\begin{equation}\label{e:DG}
 \DG (u(x),u(y)) \leq \sup_{z\in[u(x),u(y)]} |G'(z)|
  \lesssim (1 \vee |x| \vee |y|)^{\eps(m-2)}.
\end{equation}
Here in~\eqref{e:DG} we critically used the fact that $m \ge 2$
and $1 + \Psi_\eps (x) \le (1 \vee |x|)^\eps$.  Another crucial observation
that follows from $m \ge 2$ is that $G$ is convex; one has therefore
\begin{equation}\label{e:DGconvex}
\DG (a,b) \geq \DG(a',b) \geq \DG(a',b')
\end{equation}
as soon as $0\leq a'\leq a$ and $0< b'\leq b$.
\begin{lemma}
For $m\geq2$, if at least one of the values $u(x)$ or $u(y)$ is larger than $c>0$ then
\begin{equation}\label{e:DG2}
\DG (u(t,x),u(t,y)) \ge c^{m-2}.
\end{equation}
\end{lemma}
\proof If $u(x)\geq c$ and $u(y)\geq c$ then $\DG(u(x),u(y)) = G'(z)$ for some $z\in[u(x),u(y)]$
with an increasing function $G'(z)=(m-1)z^{m-2}$. One has therefore 
\[\DG(u(x),u(y))\geq (m-1) c^{m-2}\geq c^{m-2}\]
in this case. On the other hand, if $u(x)\geq c\geq u(y)\geq0$ then by the convexity of $G$,
the inequality~\eqref{e:DGconvex} implies
\[
\DG(u(x),u(y))\geq \DG(0,c)=c^{m-2}.
\]
The case $u(y)\geq c\geq u(x)\geq0$ is similar and the lemma follows.
\qed

\subsection{Common estimate for $\mathcal{E}_+$ and $\mathcal{E}_-$ and any $\alpha\in(0,2)$}

Controlling $\QG$ will require a different approach for $\mathcal{E}_+$ and for $\mathcal{E}_-$. Dealing with the first term on the right-hand side of  \eqref{e:E'start} is much easier. 
\begin{lemma}\label{l:ee1}
For $\eps < \frac\alpha{m}$, we have:  
\[ - \BG (\pm (u-\fipm)_\pm, \fipm ) \le  \frac12 \BG (\ufi,\ufi) + C_{\fipm} |\{(u-\fipm)_\pm >0\}|\]
where $C_{\fipm} \gtrsim 1+ \| \nabla \fipm\|_\infty^2 $. 
\end{lemma}
\begin{proof}
Keeping track of the support of $u_\varphi^\pm$ we write
\begin{align*}
  - \BG (\pm(u-\fipm)_\pm, \fipm ) = &  \mp \iint (\ufi (y) -\ufi (x)) (\fipm (y) -\fipm (x)) \DG (u(x),u(y)) \frac{c_{\alpha,N}^0 \dx \dy}{|y-x|^{N+\alpha}} 
\\
\le &  \frac12 \BG (\ufi,\ufi) \\
& + \frac12 \iint (\fipm (y) -\fipm (x))^2 (\un_{\ufi(x)>0} + \un_{\ufi(y)>0}) \DG (u(x),u(y)) \frac{c_{\alpha,N}^0 \dx \dy}{|y-x|^{N+\alpha}}\cdotp
\end{align*}
Thanks to~\eqref{e:DG} we estimate the second term of the right hand side as follows:
\begin{align*}
 \iint & (\fipm (y) -\fipm (x))^2 (\un_{\ufi(x)>0} + \un_{\ufi(y)>0}) \DG (u(x),u(y)) \frac{c_{\alpha,N}^0 \dx \dy}{|y-x|^{N+\alpha}} \\
& \lesssim \int_{\ufi>0} \left\{ \int (\fipm (y) -\fipm (x))^2  (1\vee |x| \vee |y|)^{\eps (m-2)} \frac{\dy}{|y-x|^{N+\alpha}} \right\} \dx. 
\end{align*}
Since $\{\ufi>0\}$ is contained in $B_{2^{\frac1\eps}}$, we have
\begin{align*}
\int (\fipm (y) -\fipm (x))^2  &(1\vee |x| \vee |y|)^{\eps (m-2)} \frac{\dy}{|y-x|^{N+\alpha}} \\
& \lesssim \|\nabla\fipm\|_{L^\infty}^2 \int_{|y-x|\le 1}  \frac{\dy}{|y-x|^{N+\alpha-2}} 
+ \|\fipm\|_{L^\infty(B_{2^{1/\eps}})}^2 \int_{|y-x|\ge 1}  |y-x|^{\eps m} \frac{\dy}{|y-x|^{N+\alpha}} \\
& \lesssim 1,
\end{align*}
provided $\eps < \alpha/m$.  This yields the desired estimate.
\end{proof}

\subsection{Estimates for $\mathcal{E}_+$ for $\alpha>1$}

We first estimate $\QG ( (u-\fip)_+, u)$. In order to do so, we  split it as follows (see \cite{CSV}):
 \[\QG ( (u-\fip)_+, u) = \Qint^{+,+}+ \Qint^{+,-}+ \Qint^{+,0}+  \Qout^{+,+}+ \Qout^{+,-}+ \Qout^{+,0}\]
 with 
 \[\begin{cases}
 \Qint^{+,+} &= \QGint ((u-\fip)_+, (u-\fip)_+) \\
 \Qint^{+,-} &= \QGint ((u-\fip)_+, -(u-\fip)_-)\\
\Qint^{+,0} &=  \QGint ( (u-\fip)_+, \fip)\\
  \Qout^{+,+}&= \QGout ( (u-\fip)_+, (u-\fip)_+) \\ 
\Qout^{+,-} &= \QGout ( (u-\fip)_+, -(u-\fip)_-)\\
 \Qout^{+,0} &= \QGout ( (u-\fip)_+, \fip).
 \end{cases}
 \]
where $\QGint$ and $\QGout$ are defined by
\begin{align*}
\QGint (v,w) = \iint_{|x-y| \le \eta} 
 v(x) (w(y)-w(x)) \DG (u(x),u(y)) \frac{\nabla \fip (x)}{\fip(x)} \cdot (y-x) \frac{c_{\alpha,N} \dx \dy }{ |y-x|^{N+\alpha}} , \\
\QGout (v,w) = \iint_{|x-y| \ge \eta} 
 v(x) (w(y)-w(x)) \DG (u(x),u(y)) \frac{\nabla \fip (x)}{\fip(x)} \cdot (y-x) \frac{c_{\alpha,N} \dx \dy }{ |y-x|^{N+\alpha}} 
\end{align*} 
for some small parameter $\eta \in (0,1)$ to be fixed later (see lemma~\ref{l:ee3+} below).  We
estimate successively the six terms appearing in this decomposition.
Note that we only need upper estimates as negative terms can be discarded from the right-hand side of \eqref{e:E'start}.

\begin{lemma}\label{l:ee2+}
For $\alpha<2\leq m$, one has
 \[
  \Qint^{+,+}  \le \frac14 \BG(\ufip,\ufip) + C_{\fip} |\{(u-\fip)>0\}| 
 \]
where $C_{\fip} \gtrsim \|\nabla \fip/\fip\|^2_\infty$. 
\end{lemma}
\begin{proof}
We first write a Cauchy-Schwarz type inequality and~\eqref{e:DG}:
\begin{align*}
\Qint^{+,+} &=  \QGint (\ufip, \ufip) \\
& = \iint_{|x-y| \le \eta} 
 \ufip(x) (\ufip(y)-\ufip(x)) \DG (u(x),u(y)) \frac{\nabla \fip (x)}{\fip(x)} \cdot (y-x) \frac{c_{\alpha,N} \dx \dy }{ |y-x|^{N+\alpha}}\\
& \le \frac14 \BG (\ufip,\ufip) +  \iint_{|x-y|\le \eta} (\ufip)^2 (x) (1 \vee |x| \vee |y|)^{\eps(m-2)} \frac{|\nabla \fip (x)|^2}{\fip^2(x)} 
\frac{c_{\alpha,N} \dx \dy }{ |y-x|^{N+\alpha-2}}\cdotp
\end{align*}
Since $\fip$ is Lipschitz continuous and $\ufip(x) \le u(x) \le (1 \vee |x|)^\eps$, we have for $m\geq2$:
\[
 \iint_{|x-y|\le \eta}  (\ufip)^2 (x) (1 \vee |x| \vee |y|)^{\eps(m-2)} \frac{|\nabla \fip (x)|^2}{\fip^2(x)} 
\frac{c_{\alpha,N} \dx \dy }{ |y-x|^{N+\alpha-2}}  \lesssim \eta^{2 - \alpha} |\{(u-\fip)>0\}|.
\]
In this integral, the variable $x$ is confined into $B_{2^{1/\epsilon}}$ and the $y$ variable is controlled by the following fact:
\begin{equation}\label{estim:vee}
\sup_{x\in B_{2^{1/\eps}}}
 \int_{|x-y|\le \eta} (1 \vee |x| \vee |y|)^{\eps(m-2)} \frac{c_{\alpha,N} \dy }{ |y-x|^{N+\alpha-2}} \lesssim \eta^{2-\alpha}
\end{equation}
since \( (1 \vee |x| \vee |y|)^{\eps(m-2)} \lesssim 1 \) if $|y-x| \le \eta < 1$ and $\alpha<2$.
\end{proof}
\begin{lemma}\label{l:ee3+}
For $\eta$ such that $\eta  \le \frac{c_{\alpha,N}^0}{2c_{\alpha,N} C_{\fip}} $, we have:
\[\Qint^{+,-} \le \frac12 \BG ((u-\fip)_+,-(u-\fip)_-).\]
Let us recall that the constants of the singular integrals
are defined by~\eqref{fractional} and \eqref{fractional_laplace}.
\end{lemma}
\begin{proof}
The term $\Qint^{+,-}$ is easy to handle. We simply write 
\begin{align*} \Qint^{+,-} &= -\QGint ((u-\fip)_+,(u-\fip)-) \\
& = - \iint_{|x-y| \le \eta}  (u-\fip)_+(x) (u-\fip)_-(y) 
\DG (u(x),u(y)) \frac{\nabla \fip (x)}{\fip(x)} \cdot (y-x) \frac{c_{\alpha,N} \dx \dy }{ |y-x|^{N+\alpha}} \\
& \le \| \nabla \fip / \fip\|_\infty \eta \iint_{|x-y| \le \eta}  (u-\fip)_+(x) (u-\fip)_-(y) \DG (u(x),u(y))  \frac{c_{\alpha,N} \dx \dy }{ |y-x|^{N+\alpha}} \\
& \le \frac12 \BG ((u-\fip)_+,-(u-\fip)_-)
\end{align*}
provided $\eta$ is chosen small enough to ensure that $\eta c_{\alpha,N} \| \nabla \fip / \fip\|_\infty  \le \frac12 c_{\alpha,N}^0$. 
\end{proof}
\begin{lemma}\label{l:ee4+}
For $\alpha<2\leq m$, one has
\[ \Qint^{+,0} \le C_{\fip} |\{(u-\fip) >0 \}|,
\] where $C_{\fip} \gtrsim \|\nabla \fip\|_\infty
\| \nabla \fip /\fip\|_\infty$. 
\end{lemma}
\begin{proof}
The proof is similar to that of lemma \ref{l:ee2+}, but this time the regularity of $\fip$ provides the local integrability, instead of using Cauchy-Schwarz:
\begin{align*}
\Qint^{+,0}  & = \iint_{|x-y|\le \eta} \ufip(x) (\fip(y)-\fip(x)) 
\DG (u(x),u(y)) \frac{\nabla \fip (x)}{\fip(x)} \cdot (y-x) \frac{c_{\alpha,N} \dx \dy }{ |y-x|^{N+\alpha}} \\
& \le \|\nabla \fip\|_\infty \| \nabla \fip /\fip\|_\infty \int_{B_{2^{1/\epsilon}}} \ufip(x)  
\left\{ \int_{|x-y|\le \eta} (1 \vee |x| \vee |y|)^{\eps(m-2)}    \frac{c_{\alpha,N}  \dy }{ |y-x|^{N+\alpha-2}} \right\} \dx \\
& \lesssim |\{(u-\fip)>0\}|.
\end{align*}
Once more, we used \eqref{e:DG}, the fact that the variable $x$ is confined into $B_{2^{1/\epsilon}}$ 
and that $\alpha<2\leq m$.
\end{proof}
\begin{lemma}\label{l:ee5+}
For $\alpha>1$ and $\epsilon<(\alpha-1)/(m-1)$, we have:
\[ \Qout^{+,+} \le C_{\fip} |\{ (u -\fip)>0\}|,
\]
with $C_{\fip} \gtrsim \|\nabla \fip / \fip\|_\infty$. 
\end{lemma}
\begin{proof}
 We use \eqref{e:DG},
 the boundedness of  $\nabla \fip/\fip$ and $\ufip (y)\le (1\vee |y|)^\eps$ in order to get
\begin{align*}
\Qout^{+,+}  \lesssim & \iint_{|x-y| \ge \eta} 
 \ufip(x) (\ufip(y)+\ufip(x)) (1 \vee |x| \vee |y|)^{\eps(m-2)} \frac{c_{\alpha,N} \dx \dy }{ |y-x|^{N+\alpha-1}} \\
 \lesssim &  \int \ufip(x) \left\{ \int_{|y-x|\ge \eta} 
 (1 \vee |y|)^{\eps(m-1)}  \frac{c_{\alpha,N}  \dy }{ |y-x|^{N+\alpha-1}}\right\} \dx \\
& +\int \ufip(x)^2 \left\{ \int_{|y-x|\ge \eta} (1 \vee |y|)^{\eps(m-2)}  \frac{c_{\alpha,N}  \dy }{ |y-x|^{N+\alpha-1}}\right\} \dx. 
\end{align*}
We use here in an essential way that $\alpha >1$ and $\epsilon<(\alpha-1)/(m-1)$
in order to get that  the two terms in braces are $\lesssim 1$. 
Note that as $m\geq2$, one has $\alpha-1>\epsilon(m-1)>\epsilon(m-2)\geq0$.
This yields the desired estimate. 
\end{proof}
\begin{lemma}\label{l:ee6+}
For $\alpha>1$ and  $\epsilon<(\alpha-1)/(m-1)$, we have:
\[ \Qout^{+,-} \le C_{\fip} |\{(u-\fip)>0\}|,
\]
with $C_{\fip} \gtrsim \|\nabla \fip /\fip\|_\infty$. 
\end{lemma}
\begin{proof}
We use \eqref{e:DG}, the boundedness of $\nabla \fip/\fip$  and $(u-\fip)_-(y) \leq \fip(y) \le (1 \vee |y|)^\eps$ in order to get
\begin{align*}
\Qout^{+,-}  \lesssim & \iint_{|x-y| \ge \eta} 
 (u-\fip)_+(x)  (u-\fip)_-(y)  (1 \vee |x| \vee |y|)^{\eps(m-2)}  \frac{c_{\alpha,N} \dx \dy }{ |y-x|^{N+\alpha-1}} \\
\lesssim & \int (u-\fip)_+(x) \left\{ \int_{|x-y| \ge \eta} (1 \vee |y|)^{\eps(m-1)} \frac{c_{\alpha,N}  \dy }{ |y-x|^{N+\alpha-1}} \right\} \dx 
\\
\lesssim & \int (u-\fip)_+(x)   \dx \\
\lesssim & \; C |\{(u-\fip)>0\}|.
\end{align*}
The integral in  braces converges because $\epsilon(m-1)<\alpha-1$.
This yields the desired estimate. 
\end{proof}
\begin{lemma}\label{l:ee7+}
For $\alpha>1$ and  $\epsilon<(\alpha-1)/(m-1)$, we have:
\[ \Qout^{+,0} \le C_{\fip} |\{(u-\fip)>0\}|,
\]
 with $C_{\fip} \gtrsim \|\nabla \fip /\fip\|_\infty$. 
\end{lemma}
\begin{proof}
We offer here a slightly simpler proof than in~\cite{CSV}.
Again, let us observe that $x$ is confined in $B_{2^{1/\epsilon}}$ in the following integral
so one can use $|\fip(y)-\fip(x)|\lesssim (1+|x|+|y|)^\epsilon$ and \eqref{e:DG}:
\begin{align*}
\Qout^{+,0}  = & \iint_{|x-y| \ge \eta} 
 \ufip(x) (\fip(y)-\fip(x)) \DG (u(x),u(y)) \frac{\nabla \fip (x)}{\fip(x)} \cdot (y-x) \frac{c_{\alpha,N} \dx \dy }{ |y-x|^{N+\alpha}} \\
 \lesssim &  \iint_{|x-y| \ge \eta}  \ufip(x)   (1 \vee |y|)^{\eps (m-1)}  \frac{ \dx \dy }{ |y-x|^{N+\alpha-1}} \\
\lesssim & \int \ufip (x) \left\{ \int_{|y-x| \ge \eta}  (1  \vee |y|)^{\eps (m-1)}  \frac{\dx \dy }{ |y-x|^{N+\alpha-1}} \right\} \dx \\
\lesssim & \enspace |\{ \ufip >0\}|.
\end{align*}
The integral in braces converges because $\epsilon(m-1)<\alpha-1$.
\end{proof}

\begin{remark}
Note that up to now, as $\alpha<2\leq m$, the most stringent condition on $\epsilon$ is
\[
0<\epsilon < \min\left\{  \frac{\alpha}{m}, \frac{\alpha-1}{m-1} \right\}
= \frac{\alpha-1}{m-1} = \epsilon_0.
\]
\end{remark}

\subsection{Estimates for $\mathcal{E}_-$ for $\alpha>1$}

In order to estimate $\QG ( -(u-\fim)_-, u)$, we split it again as follows (see \cite{CSV}), 
but we group the terms differently:
\[ \QG (-\ufim, u) = \QG (\ufim, \ufim) + \QG (-\ufim,(u-\fim)_++\fim).\]
Let us point out  that the previous sub-split, which depends on the size of $|x-y|$, will still be necessary
for each term, but the cut-off value $\eta$ will be different between the proof of lemma~\ref{l:ee2-}
and that of lemma~\ref{l:ee3-}.
\begin{lemma}\label{l:ee2-}
For $\alpha-1>\epsilon(m-2)\geq0$, we have 
\[\QG (\ufim,\ufim)  \le \frac14 \BG (\ufim,\ufim) + C C_{\fim} |\{\ufim >0 \}| .\]
\end{lemma}
\begin{proof}
We first write 
\[ \QG (\ufim,\ufim)   = \Qint + \Qout\] 
with
\begin{align*}
\Qint = & \int_{|x-y| \le \eta} \ufim (x) (\ufim (x) -\ufim (y)) \DG (u(x),u(y)) \frac{\nabla \fim (x)}{\fim(x)} \cdot (y-x) \frac{c_{\alpha,N} \dx \dy }{ |y-x|^{N+\alpha}}\\
\Qout = & \int_{|x-y| \ge \eta} \ufim (x) (\ufim (x) -\ufim (y)) \DG (u(x),u(y)) \frac{\nabla \fim (x)}{\fim(x)} \cdot (y-x) \frac{c_{\alpha,N} \dx \dy }{ |y-x|^{N+\alpha}}
\end{align*}
for some $\eta>0$ of arbitrary value.

\medskip
We argue as in the proof of lemma~\ref{l:ee2+} by writing first, thanks to \eqref{e:DG}  and the properties of $\fim$, that:
\begin{align*}
\Qint & \le \frac14 \BG (\ufim,\ufim) +  C \iint_{|x-y|\leq \eta} (\ufim)^2 (x) (1 \vee |x| \vee |y|)^{\eps(m-2)} \frac{|\nabla \fim (x)|^2}{\fim^2(x)} 
\frac{c_{\alpha,N} \dx \dy }{ |y-x|^{N+\alpha-2}} \\
& \le \frac14 \BG (\ufim,\ufim) + CC_{\fim} \int (\ufim)^2 (x) \fim^{-2/m_0} (x) \dx .
\end{align*}
Using $\ufim \le \fim$ yields the desired estimate for this term since $m_0 \ge 1$.
Note that the integral in
$\dy$ did converge because of the assumption $\alpha<2$.

\medskip
For the outer part
 we use \eqref{e:DG},  $\nabla \fim/\fim\leq C_{\fim} \fim^{-1/m_0}$
 and $\ufim (y)\le \fim(y) \le 1$ in order to get
\begin{align*}
\Qout  \lesssim & \int
 \ufim(x) \fim^{-1/m_0}(x) \left\{ \int_{|x-y| \ge \eta}  (1 \vee |x| \vee |y|)^{\eps(m-2)} \frac{c_{\alpha,N} \dy }{ |y-x|^{N+\alpha-1}} \right\}  \dx \\
 \lesssim &  
 C_{\fim} \int \ufim (x) \fim^{-1/m_0} (x) \dx \leq C_{\fim} \int_{\ufim>0} \fim^{1-1/m_0} (x) \dx\\
 \lesssim & C_{\fim} |\{\ufim >0\}|.
\end{align*}
We use here in an essential way the fact that $\alpha-1>\epsilon(m-2)\geq0$
to ensure the convergence of the $\dy$ integral in braces.
This yields the desired estimate.

\end{proof}
\begin{lemma}\label{l:ee3-}
For $\eps < \frac{\alpha-1}{m-1}$, we have 
\[\QG ((u-\fim)_+ + \fim,u) \le  -\frac12 \BG ((u-\fim)_+,(u-\fim)_-)  + C C_{\fim}  |\{ \ufim >0 \}|.\]
\end{lemma}
\begin{proof}
We first write 
\[ \QG ((u-\fim)_+ + \fim,u)   = \Qint + \Qout\] 
with $(u-\fim)_+ +\fim = u \vee \fim$ and
\begin{align*}
\Qint = & \int_{|x-y| \le \eta} \ufim (x) (u \vee \fim (x) - u \vee \fim (y)) \DG (u(x),u(y)) \frac{\nabla \fim (x)}{\fim(x)} \cdot (y-x) \frac{c_{\alpha,N} \dx \dy }{ |y-x|^{N+\alpha}}\\
\Qout = & - \int_{|x-y| \ge \eta} \ufim (x) u \vee \fim (y) \DG (u(x),u(y)) \frac{\nabla \fim (x)}{\fim(x)} \cdot (y-x) \frac{c_{\alpha,N} \dx \dy }{ |y-x|^{N+\alpha}}
\end{align*}
for some parameter $\eta >0$ to be chosen subsequently.
Let us point out that we removed the
term $u \vee \fim (x)$ for $|x-y|\ge \eta$, since it is away from the
singularity, and that the kernel is anti-symmetric, which makes the corresponding
$\dy$ integral vanish.

\medskip
Let us observe that $\ufim$ is supported in $B_2$.
Choosing $\eta$ large enough we can ensure that for $|x-y| \ge \eta$ one has
$u \vee \fim (y) = u (y)$ and consequently
\begin{align*}
\Qout = & - \int_{|x-y| \ge \eta} \ufim (x)u (y) \DG (u(x),u(y)) \frac{\nabla \fim (x)}{\fim(x)} \cdot (y-x) \frac{c_{\alpha,N} \dx \dy }{ |y-x|^{N+\alpha}} \\
\lesssim & C_{\fim} \int \ufim (x) \fim^{-1/m_0} (x) \left\{ \int_{|y-x| \ge \eta} (1 \vee |y|)^{\eps (m-1)}  \frac{\dy }{ |y-x|^{N+\alpha-1}} \right\} \dx \\
\lesssim & C_{\fim} \int \ufim (x) \fim^{-1/m_0} (x) \dx  \leq C_{\fim} \int_{\ufim>0} \fim^{1-1/m_0} (x) \dx\\
\lesssim & C_{\fim} |\{\ufim >0\}|
\end{align*}
since $m_0 \ge 1$ and $\ufim \le \fim$. 

\medskip
As far as $\Qint$ is concerned, we revert to $u \vee \fim = (u-\fim)_+ +\fim$ and split
it as $\Qint = \Qint^{-,+} + \Qint^{-,0}$ with 
\begin{align*}
\Qint^{-,+} = & \int_{|x-y| \le \eta} \ufim (x) ((u - \fim)_+ (x) - (u - \fim)_+ (y)) \DG (u(x),u(y)) \frac{\nabla \fim (x)}{\fim(x)} \cdot (y-x) \frac{c_{\alpha,N} \dx \dy }{ |y-x|^{N+\alpha}}\\
\Qint^{-,0} = & \int_{|x-y| \le \eta} \ufim (x) (\fim (x) - \fim (y)) \DG (u(x),u(y)) \frac{\nabla \fim (x)}{\fim(x)} \cdot (y-x) \frac{c_{\alpha,N} \dx \dy }{ |y-x|^{N+\alpha}}\cdotp
\end{align*}
We split the integral further, depending on the size of the unsigned factors:
\begin{align*}
\Qint^{-,+} = & - \int_{|x-y| \le \eta} (u-\fim)_- (x) (u - \fim)_+ (y) \DG (u(x),u(y)) \frac{\nabla \fim (x)}{\fim(x)} \cdot (y-x) \frac{c_{\alpha,N} \dx \dy }{ |y-x|^{N+\alpha}}\\
= & - \int_{\underset{|x-y| \le \eta}{\fim^{-1}(x) |\nabla \fim (x)| |y-x| \le 1/2}} 
(u-\fim)_- (x) (u - \fim)_+ (y) \DG (u(x),u(y)) \frac{\nabla \fim (x)}{\fim(x)} \cdot (y-x) \frac{c_{\alpha,N} \dx \dy }{ |y-x|^{N+\alpha}}\\
  & - \int_{\underset{|x-y| \le \eta}{\fim^{-1}(x) |\nabla \fim (x)| |x-y| \ge 1/2}}  
(u-\fim)_- (x) (u - \fim)_+ (y) \DG (u(x),u(y)) \frac{\nabla \fim (x)}{\fim(x)} \cdot (y-x) \frac{c_{\alpha,N} \dx \dy }{ |y-x|^{N+\alpha}}\cdotp
\end{align*}
We then use the good extra term as follows:
\begin{align*}
\Qint^{-,+}  \le & -\frac12 \BG ((u-\fim)_+,(u-\fim)_-) \\
& + C C_{\fim} \int (u-\fim)_- (x) \fim^{-1/m_0} (x)  \left\{ \int_{\fim^{1/m_0}(x) /(2 C_{\fim}) \le |x-y| \le \eta} \frac{  \dy }{ |y-x|^{N+\alpha-1}} \right\} \dx \\
\le & -\frac12 \BG ((u-\fim)_+,(u-\fim)_-)  + C C_{\fim}\int (u-\fim)_- (x) \fim^{-\alpha /m_0} (x)  \dx. 
\end{align*}
As  $m_0\geq2>\alpha$, the last integral is related to the measure of the set $\{\ufim>0\}$ in the following way:
\[
\int_{\ufim>0} (u-\fim)_- (x) \fim^{-\alpha /m_0} (x)  \dx
\leq \int_{\ufim>0} \fim^{1-\alpha /m_0} (x)  \dx
\leq C_{\fim} |\{\ufim>0\}|.
\]

Similarly, using \eqref{e:DG}, the fact that $\supp \ufim\in B_2$
and the properties of $\fim$, we get for the last term: 
\begin{align*}
\Qint^{-,0} & \lesssim C_{\fim} \int\ufim (x) \fim^{-1/m_0} (x) \left\{\int_{|x-y| \le \eta}  (1\vee |y|)^{\eps (m-2)}   \frac{\dy }{ |y-x|^{N+\alpha-2}} \right\} \dx \\
& \lesssim C_{\fim} \int\ufim (x) \fim^{-1/m_0} (x)  \dx \lesssim C_{\fim} |\{\ufim>0\}|.
\end{align*}
This yields the desired estimate. 
\end{proof}

%
%

\begin{remark}
Note that for this second half of the proof, as $\alpha<2\leq m$, the most stringent
condition on $\epsilon$ is
\[
0<\epsilon < \min\left\{  \frac{\alpha}{m}, \frac{\alpha-1}{m-1}, \frac{\alpha-1}{m-2} \right\}
= \frac{\alpha-1}{m-1} = \epsilon_0,
\]
which is the same critical value as before.
\end{remark}

\begin{remark}
Note that the value of $\eta$ (which defines the cut-off threshold between the inner and outer regions) for the estimate of $\mathcal{E}_-$ does not have to
match the value of $\eta$ for the estimate of $\mathcal{E}_+$. This is why one can require $\eta$ to be small in the proof of lemma~\ref{l:ee3+} and large
in lemma~\ref{l:ee3-}.
\end{remark}

\subsection{Modifications in the case $\alpha\leq1$}

The previous energy estimates of the terms $Q_{\text{out}}^{+,+}$,  $Q_{\text{out}}^{+,-}$,
 $Q_{\text{out}}^{+,0}$ only work for $\alpha>1$ where the faster decay of the non-local
 kernel allows the integrals to converge.
Let us now deal with the necessary modifications of the proof in the case where $\alpha\leq1$.
Instead of considering the solution $u(t,x)$, the idea is to use a drifting change of variable:
\begin{equation}\label{uDrift1}
\bar{u}(t,x)=u(t,x-s(t))
\end{equation}
with a properly chosen drift $s:(0,T)\to\R^N$.
Such a function $\bar{u}$ solves
\begin{equation}
\partial_t \bar u = {  \nabla \cdot} \left(\bar u \nabla^{\alpha-1}G(\bar u) \right)+h, \quad t >0, \quad  x \in \mathbb{R}^N
\end{equation}
where the additional forcing term is given by:
\begin{equation}
h(t,x)=-s'(t)\cdot\nabla \bar{u} = -\nabla\cdot(v(t) \bar{u})
\end{equation}
with a drift velocity $v(t)=-s'(t)\in\R^N$.

\bigskip
Because of the forcing term, the energy estimates \eqref{e:E'start} now
contain one additional term:
\begin{align} \label{e:E'startNEW}
\frac{d}{dt} \Epm (t) &+ \BG ( (\bar{u}-\fipm)_\pm, (\bar{u}-\fipm)_\pm ) - \BG ( (\bar{u}-\fipm)_+, (\bar{u}-\fipm)_-) \\
&= - \BG (\pm (\bar{u}-\fipm)_\pm, \fipm ) + \QG (\pm (\bar{u}-\fipm)_\pm, \bar{u})
+\int  H' \left(1 \pm \frac{(\bar{u}-\fipm)_\pm}{\varphi_\pm} \right)  h \dx.
\notag
\end{align}
The general idea is that the default of compactness and lack of convergence
of the kernel in $\QG$ will ultimately be compensated by that of the additional term,
provided the drift is chosen properly. For $\alpha=1$, one will also need to pay attention
to what happens near the origin as $1/|x|^N$ fails to be integrable at both ends.

\bigskip
Proceeding by integration by part as was done~p.\pageref{IPP}, the additional term is:
\begin{align*}
\int  H' \left(1 \pm \frac{(\bar{u}-\fipm)_\pm}{\varphi_\pm} \right)  h \dx
&=-\int  H' \left(1 \pm \frac{(\bar{u}-\fipm)_\pm}{\varphi_\pm} \right)  \nabla\cdot(v(t) \bar{u}) \dx\\
& = - \int H'' \left(1 \pm \frac{(\bar{u}-\fipm)_\pm}{\varphi_\pm} \right)
\nabla \left(\frac{\pm(\bar{u}-\fipm)_\pm}{\fipm} \right) \cdot (v(t) \bar{u}) \dx\\
&=  - \int \fipm \nabla \left(\frac{\pm(\bar{u}-\fipm)_\pm}{\fipm} \right)\cdot v(t)  \dx.
\end{align*}
As $\int \nabla(\pm(\bar{u}-\fipm)_\pm)\cdot v(t)=0$, it all boils down to:
\[
\int  H' \left(1 \pm \frac{(\bar{u}-\fipm)_\pm}{\varphi_\pm} \right)  h 
=\mp \int (\bar{u}-\fipm)_\pm \frac{\nabla\varphi_\pm}{\varphi_\pm} \cdot v(t). 
\]
The hot term being $\QG (\pm (\bar{u}-\fipm)_\pm, \bar{u})$ where
\[
\QG (k,w) =
\iint k(x) (w(y)-w(x)) \DG (\bar{u}(x),\bar{u}(y)) \frac{\nabla \fipm (x)}{\fipm(x)} \cdot (y-x) \frac{c_{\alpha,N} \dx \dy }{ |y-x|^{N+\alpha}},
\]
it is therefore natural to choose the drift $s(t)$ such that
\begin{equation}\label{uDrift2}
s(0)=0 \quad\text{and}\quad
v(t) = -s'(t) = c_{\alpha,N}\int \frac{y  \chi(|y|)}{|y|^{N+\alpha}}\left( G(\bar{u}(y))-G(\bar{u}(y_0)) \right) \dy,
\end{equation}
with a smooth cut-off $\chi(r)\simeq \mathbf{1}_{r>r_0}$ supported in $[r_0,\infty)$ 
and a point $y_0$ chosen arbitrarily. Note that as $\int_{|y|=c} \frac{y  \chi(|y|)}{|y|^{N+\alpha}} \dy =0$, one can adjust the choice of $y_0$ on the fly without
changing the value of $v(t)$.  For example, choosing $y_0=x$ when $v(t)$ is paired with a function evaluated at $x$, one thus gets:
\[
\int  H' \left(1 \pm \frac{(\bar{u}-\fipm)_\pm}{\varphi_\pm} \right)  h 
=\mp \iint (\bar{u}-\fipm)_\pm(x) (\bar{u}(y) -\bar{u}(x)) \DG(\bar{u}(y),\bar{u}(x))
 \frac{\nabla\varphi_\pm(x)}{\varphi_\pm(x)} \cdot y \chi(|y|)
\frac{c_{\alpha,N}  \dx\dy}{|y|^{N+\alpha}}\cdotp
\]
For $\alpha=1$, the lack of integrability at the origin makes the cut-off absolutely
necessary. For $\alpha<1$, it is just a harmless convenience, but we can make the most of it by choosing $r_0$ properly.
The former $\QGout$ term is joined with the new term and thus replaced by:
\begin{gather*}
\QGout (k,w) = \iint_{|x-y| \ge \eta} 
 k(x) (w(y)-w(x)) \DG (\bar{u}(x),\bar{u}(y)) \frac{\nabla \fipm (x)}{\fipm(x)} \cdot (y-x) \frac{c_{\alpha,N} \dx \dy }{ |y-x|^{N+\alpha}} \\
 - \iint
k(x) (w(y) -w(x)) \DG(\bar{u}(y),\bar{u}(x))
 \frac{\nabla\varphi_\pm(x)}{\varphi_\pm(x)} \cdot y \chi(|y|)
\frac{c_{\alpha,N}  \dx\dy}{|y|^{N+\alpha}}\cdotp
\end{gather*}
Note that the second integral is a-priori computed on $\R^{2N}$ and
when $|x-y|\leq \eta$, one will not be able to bound the second integral by the other ``good'' terms as we did before.
However, if one chooses
\[
r_0> 2^{1/\epsilon}+\eta+\|s(t)\|_{L^\infty},
\]
then in the second integral, one can also assume that $|x-y|\geq \eta$ because
${k(x)=(\bar{u}-\varphi_+)_+}$ is supported in $s(t)+ B_{2^{1/\epsilon}}$, $k(x)=-(\bar{u}-\fim)_-$ is supported in $s(t)+ B_2\subset s(t)+ B_{2^{1/\epsilon}}$ and $|y|>r_0$ because of the support of $\chi$. Therefore, with that choice for $r_0$, one has:
\begin{equation}\label{newQout}
\QGout (k,w) = \iint_{|x-y| \ge \eta} 
 k(x) (w(y)-w(x)) \DG (\bar{u}(x),\bar{u}(y)) \frac{\nabla \fipm (x)}{\fipm(x)} \left(
\frac{y-x}{ |y-x|^{N+\alpha}} - \frac{y \chi(|y|)}{|y|^{N+\alpha}} \right)  c_{\alpha,N} \dx \dy.
\end{equation}

\begin{remark}
Note that the drift vector $s(t)$ given by \eqref{uDrift2} is uniformly bounded on a finite time-interval by $\|u\|_{L^\infty_t(L^{m-1}_x)}$
which, by theorem~\ref{existence}, is itself controlled by $\|u_0\|_{L^{m-1}}\leq \|u_0\|_{L^\infty}^{\frac{m-2}{m-1}} \|u_0\|_{L^1}^{\frac{1}{m-1}}$.
The finiteness of this last quantity is part of our assumptions of theorem~\ref{thm:main}. This ensures that the definitions of $r_0$ and $s(t)$ are meaningful.
\end{remark}

\bigskip\noindent
For the energy estimate on $\mathcal{E}_+$, lemmas \ref{l:ee5+}, \ref{l:ee6+} and \ref{l:ee7+} can now be replaced by the following single one.
\begin{lemma}\label{l:ee567+}
For $0<\alpha\leq 1$, $m\geq 2$, and $\epsilon<\alpha/(m-1)$,
we have:
\[ \QGout ((\bar{u}-\varphi_+)_+,\bar{u}) \le C_{\fip} |\{ (\bar{u} -\fip)>0\}|\]
with $C_{\fip} \gtrsim \|\nabla \fip / \fip\|_\infty$. 
\end{lemma}
\begin{proof}
We have already pointed out that $x\in s(t)+B_{2^{1/\epsilon}}$ in \eqref{newQout}. One therefore has:
\begin{align*}
|\bar{u}(x)-\bar{u}(y)| & \leq |\bar{u}(x)| + |\bar{u}(y)|\\
&\leq C + (1\vee |y-s(t)|)^\epsilon\\
&\leq C' (1\vee |y|)^\epsilon.
\end{align*}
Using \eqref{e:DG} and the boundedness of $\nabla \varphi_+/\varphi_+$ in \eqref{newQout}, one gets
\[
\QGout ((\bar{u}-\varphi_+)_+,\bar{u}) \lesssim C_{\fip} \iint_{|x-y|\geq \eta} (\bar{u}-\varphi_+)_+(x) \cdot (1\vee |y|)^{\epsilon(m-1)}
 \left|
\frac{y-x}{ |y-x|^{N+\alpha}} - \frac{y \chi(|y|)}{|y|^{N+\alpha}} \right|  c_{\alpha,N} \dx \dy.
\]
Thanks to the compensation, one can now compute the formerly diverging $\dy$ integral when $0<\alpha\leq1$:
\[
\sup_{x\in B_{2^{1/\epsilon}}}\int_{|x-y|\geq \eta} (1\vee |y|)^{\epsilon(m-1)}  \left|
\frac{y-x}{ |y-x|^{N+\alpha}} - \frac{y \chi(|y|)}{|y|^{N+\alpha}} \right|  \dy
\leq C
\]
provided $\epsilon(m-1)<\alpha$. One thus gets
\[
\QGout ((\bar{u}-\varphi_+)_+,\bar{u}) \lesssim C_{\fip} \int (\bar{u}-\varphi_+)_+(x) \dx,
\]
which finally leads to the announced estimate as $\bar{u}(t,x)\leq 1$ when $x\in s(t)+B_{2^{1/\epsilon}}$.
\end{proof}

\bigskip\noindent
For the energy estimate on $\mathcal{E}_-$, lemmas~\ref{l:ee2-} and \ref{l:ee3-} remain valid with the following update:
\begin{lemma}\label{l:ee23-}
For $0<\alpha\leq 1$, $m\geq 2$, and $\epsilon<\alpha/(m-1)$,
we have:
\[ \QGout (-(\bar{u}-\varphi_-)_-,\bar{u}) \le C_{\fim} |\{ (\bar{u} -\fim)_->0\}|\]
with $C_{\fim}$ as in theorem~\ref{thm:EE}.
\end{lemma}
\begin{proof}
In the following computation of \eqref{newQout}, one has this time $x\in\supp(\bar{u}-\fim)_-\subset s(t) + B_2$ and
\[
|\bar{u}(x)-\bar{u}(y)|\leq C(1\vee |y|)^\epsilon.
\]
One has therefore:
\[
\QGout (-(\bar{u}-\varphi_-)_-,\bar{u}) \lesssim C_{\fim} \iint_{|x-y|\geq \eta} \bufim(x) \fim^{-1/m_0}(x)\cdot (1\vee |y|)^{\epsilon(m-1)}
 \left|
\frac{y-x}{ |y-x|^{N+\alpha}} - \frac{y \chi(|y|)}{|y|^{N+\alpha}} \right|  c_{\alpha,N} \dx \dy.
\]
As in the proof of lemma~\ref{l:ee567+}, the $\dy$ integral is now convergent and one gets:
\[
\QGout (-(\bar{u}-\varphi_-)_-,\bar{u}) \lesssim C_{\fim} \int \bufim(x) \fim^{-1/m_0}(x) \dx.
\]
The final trick is that, on $\supp\bufim$, one has $\bufim(x)\leq \bar{u}(x)\leq \fim(x)\leq 1$ and thus
\[
\QGout (-(\bar{u}-\varphi_-)_-,\bar{u}) \lesssim C_{\fim} \int_{\bufim>0} \fim^{1-1/m_0}(x) \dx \lesssim C_{\fim} \left| \{\bufim>0\} \right|
\]
as claimed.
\end{proof}

\begin{remark}
Note that in the case $0<\alpha\leq 1$ and $m\geq 2$, the previous definition of $\epsilon_0$ is replaced by
\[
\epsilon_0 = \min\left\{  \frac{\alpha}{m}, \frac{\alpha}{m-1} \right\} = \frac{\alpha}{m}\cdotp
\]
\end{remark}

\subsection{Local energy estimates}

theorem \ref{thm:EE} provides a global estimate with an embedded cut-off function $\varphi_\pm$. In the sequel, we will need a localized version with the integrals computed on balls.

\begin{proposition}[Local energy estimates] \label{prop:EE-local} Let
  us assume that $\alpha\in (0,2)$ and $m \ge 2$. We take $\eps_0>0$
  given by theorem~\ref{thm:EE}. There then exists $C>0$ (only
  depending on $N,\alpha,m$) such that for any weak solution $u$
  of~\eqref{e:main} in $(-2,0] \times \RN$ satisfying
  \eqref{psiepsilon} for some $\eps \in (0,\eps_0)$, the two following
  local energy estimates hold true (with $u$ replaced by $\bar{u}$ if $\alpha\leq1$).
\begin{itemize}
\item For any $r<R$ in $(0,2^{1/\eps})$ and $c>1/4$ and with $-2<t_1 < t_2 <0$,
one has:
\begin{align} \label{e:EE-loc+}
\nonumber
\int_{B_r} (u(t_2,x)-c)_+^2\dx &
+ \int_{t_1}^{t_2} \left(\int_{B_{r}}(u-c)_+^p(x) \dx\right)^\frac{2}{p} \dt \\
+ \int_{t_1}^{t_2} & \iint_{B_r\times B_r} (u(t,x)-c)_+
\left(G(c)-G(u(y))\right)_+ \frac{\dx\dy}{|x-y|^{N+\alpha}} \, \dt
\nonumber  \\ 
&\lesssim \int_{B_R} (u(t_1,x)-c)_+^2 \dx
+ C (R-r)^{-2} \left|\{u > c \}\cap (t_1,t_2)\times B_{R} \right|.
\end{align}
\item For any cut-off function $\fim$ such that $\varphi_- \equiv 0$
  outside $B_2$ and $\varphi_- \equiv c>0$ in $B_r$ with
  \(|\nabla \fim /\fim | \le C_{\fim} \fim^{-1/m_0}\)
  for some $m_0 \ge 2$, we have
\begin{align}
\nonumber \int_{\RN} \left(u(t_2,x)-\fim (x) \right)_-^2\,& \fim^{-1}(x) \dx 
+ \int_{t_1}^{t_2} \left( \int_{B_{r}}(u(t,x)-c)_-^\frac{p m}{2} \dx\right)^\frac{2}{p} \dt \\
\label{e:EE-loc-}
\lesssim  \int_{\RN}  &\left(u(t_1,x)-\fim\right)_\pm^2 \fim^{-1}(x) \dx  
+ C_{\fim} \,  \left|\{(u - \fim)_- > 0 \}\cap (t_1,t_2)\times\RN \right|.
\end{align}
\end{itemize}
\end{proposition}
\begin{remark}
In this proposition, $p=2N/(N-\alpha)$ is given by the Sobolev embedding \eqref{sobolev}.
\end{remark}
\begin{remark}
The lower bound $c>1/4$ in~\eqref{e:EE-loc+}
is a direct inheritance from the restriction $\fip>1/4$ in theorem~\ref{thm:EE},
which in turn was constrained by the range on which the $L^2$ norm is equivalent to the
alternate energy functional~\eqref{e:alternate}.
\end{remark}
\begin{remark}
Note that the $L^p$ or $L^{pm/2}$ controls of $(u-c)_\pm$ in~\eqref{e:EE-loc+}-\eqref{e:EE-loc-}
are produced by the coercive term in~\eqref{e:EE}. The good extra term
appears as the third term on the left-hand side of~\eqref{e:EE-loc+}.
The good extra term in~\eqref{e:EE-loc+}, can be replaced for $0<\tilde{c} < c$ by
\begin{align}\label{betterGET}
 (G(c) - G(\tilde{c})) \int_{t_1}^{t_2} &
 \left( \iint_{B_r \times B_r} (u-c)_+ (x)  \un_{\{u(y) \le \tilde{c} \}}  \dx \dy  \right) \dt \\
\notag
&\lesssim \int_{B_R} (u-c)_+^2(t_1,x) \dx
+ C (R-r)^{-2} \left|\{u > c \}\cap (t_1,t_2)\times B_{R} \right|.
\end{align}
\end{remark}

\begin{proof}
We first prove \eqref{e:EE-loc+}. 
We  follow \cite{CSV} by applying the energy estimates from theorem~\ref{thm:EE} with the cut-off function 
\[ \fip (x) = 1 + \Psi_\eps(x) - (1-c) \xi (x)\]
where $\xi$ is a smooth characteristic function such that $\xi\equiv1$ on $B_r$ and
$\supp\xi\subset B_R$. 
Remark that this cut-off function satisfies the assumptions of
theorem~\ref{thm:EE} with $C_{\fip} \simeq (R-r)^{-2}$. Moreover, $\fip(x) \equiv c$ for
$x \in B_{r}$. One can apply~\eqref{e:DG2} to bound $\DG$ from below on the complementary set of $\{(x,y) \,; u(x)\vee u(y)\leq c\} $ on which the following integrand vanishes anyway. Thanks to the Sobolev embedding~\eqref{sobolev}, one thus gets the following:
\begin{align*} 
\iint ((u-c)^+(y) &-(u-c)^+(x))^2 \DG (u(t,x),u(t,y)) \frac{\dx \dy }{|y-x|^{N+\alpha}} \\
& \gtrsim c^{m-2}
\iint_{B_{r} \times B_{r}} ((u-c)_+ (y) - (u-c)_+(x))^2 \frac{\dx \dy }{|y-x|^{N+\alpha}} \\
& \gtrsim  c^{m-2}\left(\int_{B_{r}}(u-c)_+^p(x) \dx\right)^\frac{2}{p}.
\end{align*}

As far as the good extra term is concerned, we use the convexity inequality
\eqref{e:DGconvex} to assert that:
\[\begin{cases}
\text{ for } u(x) \ge c, &  \DG (u(x),u(y)) \ge \DG (c,u(y)), \\
\text{ for } u(x) \ge c> \tilde{c}\ge u(y),\quad\phantom{.} & (u(y)-c)_- \DG (c,u(y)) = G(c) - G(u(y)) \ge G(c) - G (\tilde{c})
\end{cases}
\]
and in particular
\begin{align*}
\iint& (u(t,x) - \fip(x))_+ (u(t,y) - \fip(y))_- 
\DG (u(t,x),u(t,y)) \frac{\dx \dy }{|y-x|^{N+\alpha}}\\
&\gtrsim  \iint_{B_r\times B_r} (u(t,x)-c)_+
\left(G(c)-G(u(y))\right)_+ \frac{\dx\dy}{|x-y|^{N+\alpha}}
\\
&\qquad\gtrsim
(G(c) - G (\tilde{c}))
\iint_{B_r\times B_r} (u(t,x) -c)_+  \, \un_{\{u(y) \le \tilde{c} \}}
\dx \dy.
\end{align*}
For the last estimate, we discarded the denominator because $|x-y|^{-N-\alpha} \gtrsim 1$
if $x,y \in B_r$.
Applying \eqref{e:EE} from theorem~\ref{thm:EE}  then yields the first desired estimate~\eqref{e:EE-loc+}. In particular,~\eqref{betterGET} holds too.

\bigskip
We now turn to the proof of \eqref{e:EE-loc-}.
Because $m$ can be different from $2$, the dissipation term
\( \BG (\ufim,\ufim) \)
appearing in \eqref{e:EE} is treated in a slightly different way than in~\cite{CSV}.
Let us recall that $\ufim = (u-\fim)_- = u \wedge \fim - \fim$ and write
\begin{align*}
 \BG (\ufim,\ufim) & \ge \iint_{B_{r} \times B_{r}} ((u(y)-c)_-- (u(x)-c)_-)^2 \DG (u(x),u(y)) \frac{\dx \dy}{|y-x|^{N+\alpha}} \\
& \ge \iint_{B_{r} \times B_{r}} (u(y) \wedge c - u(x) \wedge c)^2 \DG (u(x),u(y)) \frac{\dx \dy}{|y-x|^{N+\alpha}}\cdotp \\
\intertext{Using the convexity inequality~\eqref{e:DGconvex}, one gets:}
 \BG (\ufim,\ufim) & \ge \iint_{B_{r} \times B_{r}} (u(y) \wedge c - u(x) \wedge c)^2 \DG (u(x)\wedge c ,u(y) \wedge c) \frac{\dx \dy}{|y-x|^{N+\alpha}} \\
& \ge \iint_{B_{r} \times B_{r}} \bigg(u(y) \wedge c - u(x) \wedge c\bigg) \bigg(G (u(x)\wedge c)-G (u(y) \wedge c) \bigg) \frac{\dx \dy}{|y-x|^{N+\alpha}} \\
& \gtrsim \iint_{B_{r} \times B_{r}} \bigg((u(y) \wedge c)^{m/2} - (u(x) \wedge c)^{m/2} \bigg)^2  \frac{\dx \dy}{|y-x|^{N+\alpha}}. 
\end{align*}
For the last inequality, we used a well-known identity~\eqref{SV}
that we recall in the appendix of this paper.
Applying the Sobolev embedding~\eqref{sobolev}, we finally get 
\[ \BG (\ufim,\ufim) \gtrsim \left( \int_{B_{r}}(u\wedge
  c)^\frac{p m}{2}(x) \dx\right)^\frac{2}{p}.\]
In particular, theorem~\ref{thm:EE} implies that for all $-2<t_1 < t_2 <0$, 
\begin{align*}
\nonumber \int_{\RN} & \left(u(t_2,x)-\fim (x) \right)_-^2\, \fim^{-1} (x) \dx  + \int_{t_1}^{t_2} \left( \int_{B_{r}}(u\wedge
  c)^\frac{p m}{2}(x) \dx\right)^\frac{2}{p} \dt \\
 \lesssim &  \int_{\RN}  \left(u(t_1)-\fim\right)_\pm^2 \fim^{-1} \dx  
+ C_{\fim} \,  \left|\{(u - \fim)_- > 0 \}\cap (t_1,t_2)\times\RN \right|.
\end{align*}
Using next that $(u-c)_-^m=(c - u \wedge c)^m \lesssim c^m + ( u \wedge c)^m$ we can play around with the Lebesgue norm:
\begin{align*}
\left(\int_{B_{r}}(u-c)_-^\frac{p m}{2}(x) \dx\right)^{\frac{2}{p}} &\lesssim
\left\| c^m\un_{u(x)<c}+(u(x)\wedge c)^m \right\|_{L^{2/p}(B_r)}\\
&\leq c^m |\{u<\fim\}\cap B_r|^{2/p}
+ \left(\int_{B_{r}}(u\wedge c)^\frac{p m}{2}(x) \dx\right)^{\frac{2}{p}}.
\end{align*}
We thus get the desired estimate \eqref{e:EE-loc-}. 
\end{proof}

\section{First lemmas of De Giorgi}\label{s:DG}

This section is devoted to the first lemmas of De Giorgi. These lemmas are
concerned with reducing the oscillation of the solution, provided $u$ spends ``most'' of
the space-time $Q=(-2,0]\times B_2$ either on the upper side or on the lower side of the a-priori range $[0,\sup_Q u]$.
Depending on wether the maximum is lowered or the infimum is increased, we get two lemmas.

\medskip
Let us define some common notations that will be used in both proofs of lemmas~\ref{l:DG-pull}
and \ref{l:DG1-low}.
For $k \in \mathbb{N},$ let us define \(T_k = -1-\frac{1}{2^k}\)
and \(r_k = 1+\frac{1}{2^k}\). One thus has an increasing sequence of times
\[-2=T_0< T_1 < T_2 < \ldots <  T_k<  \ldots < T_\infty = -1\]
and a decreasing sequence of balls:
\[
B_2=B_{r_0}\supset
B_{r_1} \supset B_{r_2} \supset \ldots \supset B_{r_k} \supset \ldots \supset B_{r_{\infty}}=B_1.
\]
The idea is to apply recursively the local energy estimates from proposition~\ref{prop:EE-local} 
with well chosen cut-off values. The sequence of nested estimates then provides, for some $c>0$,
that
\[
\int_{(-1,0]\times B_1} (u-c)_\pm^2  \dx \dt =0,
\]
which either means, depending on each respective case, that $\sup\limits_{(-1,0]\times B_1} u<c$ \quad or  $\inf\limits_{(-1,0]\times B_1} u>c$.

\subsection{Lowering the maximum}

\begin{lemma}[Lowering the maximum] \label{l:DG1-low} Let
  $\alpha \in (0,2)$.  
 For any~$\overline{\mu} \in (0,1)$, there exists $\overline{\delta} \in (0,1)$
 such that for any function $u$ that satisfies the three assumptions:
   \begin{enumerate}
  \item $u$ is locally bounded from above in the following way:
   \begin{equation}\label{e:localboundDGlow}
   \forall (t,x)\in(-2,0] \times B_{2},
   \qquad  u (t,x) \leq 1,
\end{equation}
\item the upper local energy-inequality  \eqref{e:EE-loc+} is satisfied,
\item $u$ is ``mostly'' low-valued in the sense that
\begin{equation}\label{e:mostlylow}
\left| \left\{u \leq \textstyle\frac{1+\overline{\mu}}{2} \right\} \cap (-2,0]\times B_2 \right|
\geq (1-\overline\delta) \cdot \left|(-2,0] \times B_2\right|,
\end{equation}
\end{enumerate}
then  \(u(t,x) \leq \frac{3+\overline{\mu}}{4} \text{ in } (-1,0]\times B_1.\)
\end{lemma}

\begin{remark}
Thanks to proposition~\ref{prop:EE-local}, weak solutions of~\eqref{e:main}
that satisfy the mild growth assumption~\eqref{psiepsilon}
will automatically satisfy the first two assumptions of lemma~\ref{l:DG1-low}.
It is interesting to point out that the PDE is not directly responsible for lemma~\ref{l:DG1-low}
 and that only the local energy inequality matters.
We do not require $u$ to be non-negative;
only~\eqref{e:localboundDGlow} is necessary.
Moreover, the ``good extra term'' in~\eqref{e:EE-loc+} is not required either.
\end{remark}
\begin{remark}
The admissible values for $\overline\delta$ form an interval $(0,\overline\delta_\ast)$
where $\overline\delta_\ast$ is an increasing function of $\overline\mu$.
\end{remark}
\begin{remark}
We will only use~lemma~\ref{l:DG1-low} in the proof of its improved version, lemma~\ref{l:DG1-low-improved}. We will need to use a high threshold value for $\overline\mu$ (\textit{i.e.} very close to 1).
\end{remark}

\begin{proof}
Let us use the common definition for $T_k$ and $r_k$ from the beginning of \S\ref{s:DG}.
We now define an increasing sequence (the fact that it is increasing is crucial)
\[
c_k =\frac{3+\overline{\mu}}{4} - \frac{1-\overline{\mu}}4 \frac1{2^k} \in \left[\frac{1+\overline\mu}{2},\frac{3+\overline\mu}{4}\right]
\]
and consider the quantity
\begin{equation}
U_k = \underset{t\in[T_k,0]}{\sup}\int_{B_{r_k}} (u-c_k)_+^2(t,x) \dx.
\end{equation}
To study the asymptotic behavior of the sequence $(U_k)_{k\in\mathbb{N}}$, we establish a
recurrence inequality.
We apply the local upper energy estimate \eqref{e:EE-loc+}
with $r=r_{k}$ and $R=r_{k-1}$ so that $(R-r)^{-2}=4^k$.
Note that $c_k\geq c_0>1/4$.
For all $T_{k-1}\leq t_1 \leq T_k < t_2 < 0$, we get:
\begin{align*}
\int_{B_{r_k}} (u-c_k)_+^2(t_2,x) \dx &+ \int_{t_1}^{t_2} \left(\int_{B_{r_k}}(u-c_k)_+^p(x) \dx\right)^\frac{2}{p} \dt   \\ 
&\lesssim \int_{B_{r_{k-1}}} (u-c_k)_+^2(t_1,x) \dx + 4^k \left|\{u > c_k\}\cap (t_1,t_2)\times B_{r_{k-1}} \right|.
\end{align*}
In particular, $U_k$ satisfies
(choose a time $t_1 $ that realizes the following infimum and $t_2$ that realizes $U_k$):
\[
U_k \leq \underset{t\in[T_{k-1},T_k]}{\inf}\int_{B_{r_{k-1}}} (u-c_k)_+^2(t,x) \dx + 4^k \int_{T_{k-1}}^0\int_{B_{r_{k-1}}} \un_{\{u(t,x)>c_k\}}\dx\dt.
\]
We remark that, by positivity of the integral:
\begin{align*}
\underset{t_1\in[T_{k-1},T_k]}{\inf}\int_{B_{r_{k-1}}} (u-c_k)_+^2(x,t_1)\dx &\leq \frac{1}{T_k-T_{k-1}} \int_{T_{k-1}}^{T_k} \int_{B_{r_{k-1}}} (u-c_k)_+^2(x,t_1)\dx\mathrm{d}t_1 \\
&\leq 2^k \int_{T_{k-1}}^{0} \int_{B_{r_{k-1}}} (u-c_k)_+^2(x,t_1)\dx\mathrm{d}t_1
\end{align*}
and as $(u-c_k)_+\leq u(x)\leq 1$ on $B_{r_{k-1}}\subset B_{2}$, it is bounded  by
the characteristic function:
\[
\underset{t_1\in[T_{k-1},T_k]}{\inf}\int_{B_{r_{k-1}}} (u-c_k)_+^2(x,t_1)\dx
\leq 2^k \int_{T_{k-1}}^{0} \int_{B_{r_{k-1}}} \un_{\{u(t,x)>c_k\}}\dx\mathrm{d}t.
\]
Let us point out that this is the only point in the proof where the local boundedness assumption~\eqref{e:localboundDGlow} will be used.
We have thus obtained that
\begin{equation}\label{e:Uktomeasure0}
U_k \lesssim 4^k \int_{T_{k-1}}^{0} \int_{B_{r_{k-1}}} \un_{\{u(t,x)>c_k\}}\dx\mathrm{d}t.
\end{equation}
Moreover, as the sequence $c_k$ is increasing, we note that
\[
(u-c_k)_+ > 0 \qquad\Longrightarrow\qquad
 (u-c_{k-1})_+ > c_k-c_{k-1} \geq 2^{-k}\left(\frac{1-\overline\mu}{4}
\right)>0,
\]
which transforms~\eqref{e:Uktomeasure0} into
\begin{equation}\label{e:Uktomeasure}
U_k \lesssim 4^k \int_{T_{k-1}}^{0} \int_{B_{r_{k-1}}} \un_{\left\{(u-c_{k-1})_+ > 2^{-k}\left(\frac{1-\overline{\mu}}{4}\right)\right\}} \dx\dt.
\end{equation}
Now we take $\theta = \frac{2}{p}$ and $q = 2(1-\theta)+p\theta$. 
Then, using the Markov and H\"older inequalities, we get
\begin{align*}
\int_{B_{r_{k-1}}} &\un_{\left\{(u-c_{k-1})_+ > 2^{-k}\left(\frac{1-\overline{\mu}}{4}\right)\right\}} \dx \\
&\leq \frac{4^q 2^{qk}}{(1-\overline{\mu})^q} \int_{B_{r_{k-1}}} (u-c_{k-1})_+^q \dx \\
&\leq \frac{4^q 2^{qk}}{(1-\overline{\mu})^q} \left(\int_{B_{r_{k-1}}}(u-c_{k-1})_+^2 \dx\right)^{1-\theta} \left(\int_{B_{r_{k-1}}}(u-c_{k-1})_+^p \dx\right)^\theta.
\end{align*}
Integrating in time $t$ along the interval $[T_{k-1},0]$, we get:
\begin{equation}
U_k \leq C^k\left(\underset{t\in[T_{k-1},0]}{\sup}\int_{B_{r_{k-1}}}(u-c_{k-1})_+^2(t)\dx\right)^{1-\theta} \left(\int_{T_{k-1}}^0 \left(\int_{B_{r_{k-1}}} ((u-c_{k-1})_+^p \dx\right)^\frac{2}{p} \dt\right).
\end{equation}
To control the last factor, we apply~\eqref{e:EE-loc+} one last time, but on the time interval
$t_1=T_{k-1}$ and $t_2=0$ and with the radii $r=r_{k-1}$ and $R=r_{k-2}$; we get:
\begin{equation}\label{evenmoresubtle}
\int_{T_{k-1}}^0 \left(\int_{B_{r_{k-1}}} ((u-c_{k-1})_+^p \dx\right)^\frac{2}{p} \dt
\lesssim  U_{k-1} + C^k  \int_{T_{k-1}}^{0} \int_{B_{r_{k-2}}} \un_{\{u(t,x)>c_{k-1}\}}\dx\mathrm{d}t.
\end{equation}
The measure term in~\eqref{evenmoresubtle} cannot be removed, but it is harmless.
Indeed, let us define
\begin{equation}
M_k = \int_{T_{k-1}}^{0} \int_{B_{r_{k-1}}} \un_{\{u(t,x)>c_k\}}\dx\mathrm{d}t.
\end{equation}
So far, thanks to \eqref{e:Uktomeasure0}-\eqref{evenmoresubtle}, we have established that
for any $k\geq1$:
\begin{equation}
\begin{cases}
U_k \leq C^k M_k,\\
M_k \leq C^k U_{k-1}^{1-\theta} (U_{k-1}+M_{k-1}).
\end{cases}
\end{equation}
Therefore, we have $M_k \leq \widetilde{C}^k M_{k-1}^\sigma$
with $\sigma=2-\theta=2-\frac{2}{p}>1$.

Solving the recurrence equation, we get constants $\bar{C}>1$ and $C'=\widetilde{C}^{-\sigma (\sigma-1)^{-2}}>0$ such that
\[
0\leq M_k \leq \bar{C}^{k} (C' M_1)^{\sigma^k} \underset{k\to\infty}{\longrightarrow} 0
\]
provided $M_1<1/C'$ is small enough. In turn, this estimate also implies that $U_k\to 0$.

\medskip
Using the last assumption \eqref{e:mostlylow} and the fact that $c_1>c_0=\frac{1+\overline\mu}{2}$,
we get the final control
\begin{equation}\label{choiceoverlinedelta}
M_1 = \int_{-2}^{0} \int_{B_{2}} \un_{\{u(t,x)>c_1\}}\dx\dt
 \leq \left|\{u>(1+\overline{\mu})/2\}\cap (-2,0)\times B_2\right| < \overline{\delta} \cdot
|(-2,0)\times B_2|,
\end{equation} 
which can be made arbitrarily small for a proper choice of $\overline\delta$.
Adjusting the value of $\overline\delta$ properly in \eqref{choiceoverlinedelta},
we obtain that $U_\infty = 0$, which in turn implies that $u\leq\overline{\mu}$ in $(-1,0)\times B_1$.
This achieves the proof of this first De Giorgi lemma about lowering the maximum.
\end{proof}

\subsection{Increasing the infimum}

\begin{lemma}[Increasing the infimum] \label{l:DG-pull} Let
  $\alpha \in (0, 2)$ and $m \ge 2$. For any~$\underline{\smash\mu} \in (0,1)$, there exists
  $\underline{\delta} > 0$ such that for any function $u$ that satisfies the three assumptions:
  \begin{enumerate}
  \item $u\geq0$ on $\RN$,
\item the lower local energy-inequality \eqref{e:EE-loc-} is satisfied (with the chosen value for $m$),
\item $u$ is ``mostly'' high-valued in the sense that
\begin{equation}\label{e:mostlyhigh}
\left| \left\{u \geq \textstyle\frac{1+2\underline{\smash\mu}}{3} \right\} \cap (-2,0]\times B_2 \right|
\geq (1-\underline\delta) \cdot \left|(-2,0] \times B_2\right|,
\end{equation}
\end{enumerate}
then \(u(t,x) \geq \underline{\smash\mu}\) in  \((-1,0]\times B_1.\)
\end{lemma}
\begin{remark}
Again, thanks to proposition~\ref{prop:EE-local}, weak solutions of~\eqref{e:main}
that satisfy the mild growth assumption~\eqref{psiepsilon}
will automatically satisfy the first two assumptions of lemma~\ref{l:DG-pull}.
In lemma~\ref{l:DG-pull}, we do not require $u$ to be bounded from above, nor to have a mild
growth at infinity. The non-negativity assumption is sufficient.
Again, no ``good extra term'' is required in~\eqref{e:EE-loc-} either.
\end{remark}
\begin{remark}
The admissible values for $\underline\delta$ form an interval $(0,\underline\delta^\ast)$
where $\underline\delta^\ast$ is an increasing function of $\underline{\smash\mu}$.
\end{remark}
\begin{remark}
We have chosen to state~\eqref{e:mostlyhigh} such that the cut-off value
$\frac{1+2\underline{\smash\mu}}{3}\in (1/3,1)$. Subsequently, we will
only use lemma~\ref{l:DG-pull} in the final proof of the main theorem, where
we intend to use it with $\frac{1+2\underline{\smash\mu}}{3}\geq\frac{1}{2}$
\textit{i.e.} for~$\underline{\smash\mu}\geq1/4$.
\end{remark}

\begin{proof}
We use the common definition for $T_k$ and $r_k$ from the beginning of \S\ref{s:DG}.
To apply  \eqref{e:EE-loc-}, the key is to choose the sequence of cut-off
functions $\fim=\fik$ wisely.
Following~\cite{CSV}, we define a decreasing sequence
\[
c_k = \underline{\smash\mu} + \frac{1-\underline{\smash\mu}}3 \frac{1}{2^{k}} \in \left[\underline{\smash\mu}, \frac{1+2\underline{\smash\mu}}{3}\right]
\]
and will choose $\fik \to \underline{\smash\mu} \un_{B_1}$ as $k \to +\infty$ while
ensuring, for all $k \ge 0$, that:
\[
\begin{cases} \fik \equiv c_k \text{ in } B_{r_k},\\
  \fik \equiv 0 \text{ outside } B_{r_{k-1}}
\end{cases}
\qquad\text{and}\qquad |\nabla \fik/ \fik| \le C^k \fik^{-1/m_0}.
\]
For $k=0$, we set $r_{-1}=3$ (note that $3\leq 2^{1/\epsilon_0}$ with
$\epsilon_0$ from theorem~\ref{thm:EE}, if $\epsilon_0\leq\frac{\log 2}{\log 3}$) so that $\varphi_0 \equiv \textstyle\frac{1+2\underline{\smash\mu}}{3}$ on~$B_2$ with compact support in $B_{3}$.
The critical properties of $\fik$ are visible in the graph below.
\begin{center}
\includegraphics[width=250pt]{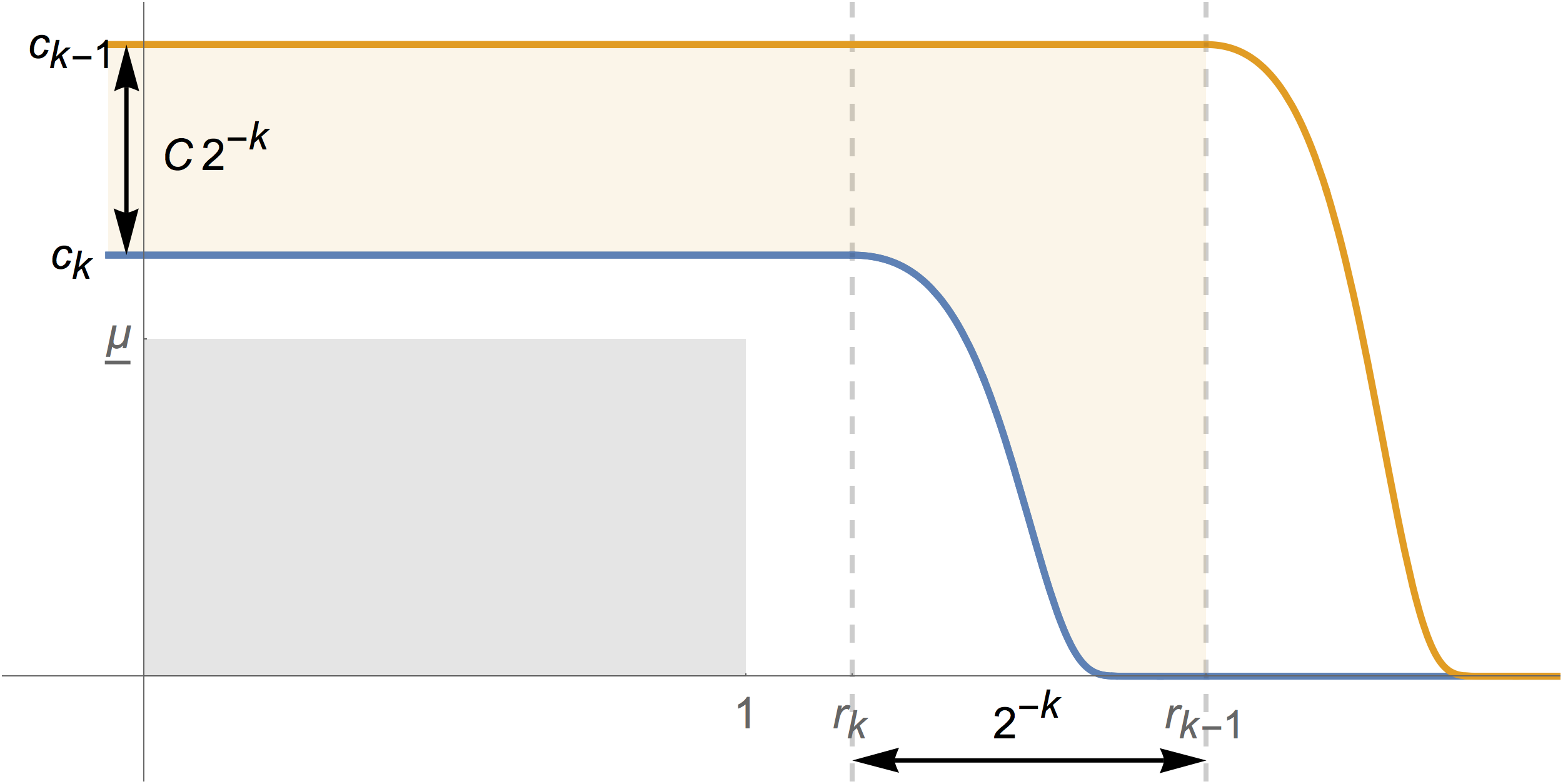}
\end{center}

\medskip
Similarly to what we did in the proof of lemma~\ref{l:DG1-low}, let us define:
\begin{equation}
V_k = \sup_{t \in [T_k,0]} \int_{\RN}  \left(u(t,x)-\fik (x) \right)_-^2\, \fik^{-1} (x) \dx.
\end{equation}
We apply the assumption \eqref{e:EE-loc-} between a starting time $t_1\in[T_{k-1},T_k]$ such that
\[
\int_{\RN}  \left(u(t_1)-\fik \right)_-^2\, \fik^{-1}
= \inf_{t\in[T_{k-1},T_k]} \int_{\RN}  \left(u(t)-\fik \right)_-^2\, \fik^{-1}
\]
and a final time $t_2\in [T_k,0]$ that realizes $V_k$.
As $u\ge0$, the function $(u-\fik)_-$ is supported in $\supp\fik\subset B_{r_{k-1}}$
and as $\fik\leq 1$, we also have $(u-\fik)_-^2 \fik^{-1} \leq \fik \leq 1$ (note that it is the only
point in the proof where we use the first assumption).
In particular, we get as in the proof of lemma~\ref{l:DG1-low}:
\begin{align*}
\inf_{t\in[T_{k-1},T_k]} \int_{\RN}  \left(u(t)-\fik \right)_-^2\, \fik^{-1}
&\leq \frac{1}{T_k-T_{k-1}}  \int^{T_k}_{T_{k-1}}\int_{\RN}  \left(u(t)-\fik \right)_-^2\, \fik^{-1} \dx\dt\\
&\leq 2^k \int_{T_{k-1}}^{0} \int_{B_{r_{k-1}}} \un_{\{u(t,x)<\fik\}}\dx\dt.
\end{align*}
The measure of $\{u<\fik\}\cap (t_1,t_2)\times\RN$ is also obviously bound by
the same right-hand side.
Thus, for this pair of times, assumption~\eqref{e:EE-loc-} implies:
\begin{equation}\label{e:Vktomeasure00}
V_k \lesssim
2^k \int_{T_{k-1}}^{0} \int_{B_{r_{k-1}}} \un_{\{u(t,x)<\fik\}}\dx\dt.
\end{equation}
As the sequence $\fik$ is decreasing both in amplitude and support in a coordinated way,
we get:
\[
\forall x\in B_{r_{k-1}}, \qquad
\fik(x) \le \fikmun(x) - \left(\frac{1-\underline{\smash\mu}}{3}\right) 2^{-k}
\]
and in particular
\[
u(t,x)<\fik \qquad\Longrightarrow\qquad (u(t,x)-\fikmun)_- > \left(\frac{1-\underline{\smash\mu}}{3}\right)2^{-k}.
\]
We are thus allowed to rewrite \eqref{e:Vktomeasure00} into
\begin{equation}\label{e:Uktomeasure11}
V_k \lesssim C^k \int_{T_{k-1}}^{0} \int_{B_{r_{k-1}}} \un_{\{(u - \fikmun)_- > (1-\underline{\smash\mu})2^{-k}/3  \}} \dx\dt.
\end{equation}
Now we take $\theta = \frac{2}{p}$ and $q = 2 (1-\theta)+\theta (pm/2)$ and apply
the Markov inequality to~\eqref{e:Uktomeasure11}, then the H\"older inequality
in the space variable and  subsequently integrate in time; we get
\begin{align*}
V_k & \le C^k \int_{T_{k-1}}^0 \int_{B_{r_{k-1}}} (u-\fikmun)_-^q  \dx \dt \\
& \le C^k \left(\sup_{t\in [T_{k-1},0]} \int_{B_{r_{k-1}}} (u-\fikmun)_-^2 \fikmun^{-1} \dx \right)^{1-\theta}
\left( \int_{T_{k-1}}^0 
\left( \int_{B_{r_{k-1}}} (u-\fikmun)_-^{\frac{pm}2}  \dx  \right)^{\frac{2}{p}}
\dt \right).
\end{align*}
Note that $\fikmun^{-1} \equiv c_{k-1}^{-1}\geq 1$ on $B_{r_{k-1}}$ so we can add it freely at the end of the computation.
Finally, let us apply~\eqref{e:EE-loc-} one more time, but between $t_1=T_{k-1}$ and $t_2=0$
and with the truncation $\fikmun$. We then get:
\begin{equation}\label{e:subt}
  \int_{T_{k-1}}^{0} \left( \int_{B_{r_{k-1}}}(u-  c_{k-1})_-^\frac{p m}{2} \dx\right)^\frac{2}{p} \dt \lesssim
  V_{k-1} + 
  C^k \int_{T_{k-1}}^{0} \int_{B_{r_{k-2}}} \un_{\{u(t,x)<\fikmun\}}\dx\dt.
\end{equation}
Roughly speaking, if we discard the measure term,
the flavor of this recurrence equation is  $V_k \le C^k V_{k-1}^{1-\theta} V_{k-1}$.
However, as there is no hope to control $\un_{\{u(t,x)<\fikmun\}}$ by $(u-\fikmun)_-$, we have to
consider the recurrence equation as a system. For this purpose, let us define
\begin{equation}
N_k=\int_{T_{k-1}}^{0} \int_{B_{r_{k-1}}} \un_{\{u(t,x)<\fik\}}\dx\dt.
\end{equation}
What we have proved so far with~\eqref{e:Vktomeasure00}-\eqref{e:subt} is the existence of a universal constant $C$ such that:
\begin{equation}
\begin{cases}
V_k \leq C^k N_k\\
N_k \leq C^k V_{k-1}^{1-\theta} (V_{k-1}+N_{k-1}).
\end{cases}
\end{equation}
From this system, we can infer that $N_k \leq \widetilde{C}^k N_{k-1}^{2-\theta}$.
Provided that $N_1$ is small enough, we will then get, as in the proof of lemma~\ref{l:DG1-low}, that $N_k\to 0$ super-exponentially fast as $k\to \infty$ (namely  $N_k\leq  \bar{C}^{k} (C' N_1)^{\sigma^k}$ with $\sigma=2-\theta>1$ so $\sigma^k \gg k$, and $C'N_1<1$) and therefore $V_k\to0$ too.

\medskip
Let us check that $N_1$ is indeed small enough.
As $\varphi_1\leq \varphi_0=\frac{1+2\underline{\smash\mu}}{3}$ on $B_2$,
our assumption \eqref{e:mostlyhigh} allows us to write
\begin{equation}\label{choiceunderlinedelta}
N_1 = \int_{-2}^0 \int_{B_2} \un_{\{u<\varphi_1\}} \dx\dt
\leq 
 \left|\{u<(1+2\underline{\smash\mu})/3\}\cap (-2,0)\times B_2\right| < \underline{\delta} \cdot
|(-2,0)\times B_2|
\end{equation} 
so it can be made arbitrarily small for a proper choice of $\underline\delta$.
This achieves the proof of the De Giorgi lemma about increasing the infimum.
\end{proof}

\section{Lemma on intermediate values}

To prove the Hölder regularity of the weak solution, we need to
improve lemma \ref{l:DG1-low} by showing that
a uniform reduction of the maximum on a smaller ball
can be obtained not only if $u$ is below 1/2 for most
of the space-time domain $(-2,0]\times B_2$ but 
that it is also true under the milder assumption that it happens for only a few
events~$(t,x) \in (-2,0]\times B_2$. 

\begin{remark}
This type of result on intermediate values is sometimes
called a ``second De Giorgi lemma'' (e.g. in~\cite{CCV})
in reference to the historical papers of E.~De Giorgi on elliptic PDEs.
As we have already established two De Giorgi lemmas of the first kind,
it would probably be more proper to call it ``a De Giorgi lemma of the second kind''.
\end{remark}

\begin{lemma}[Intermediate values, or De Giorgi lemma of the second kind]
  \label{lvi} Let $\alpha \in (0,2)$.
 For any $\rho \in (0,\frac{1}{8})$ and $\delta\in (0,\frac12)$,
 there exist $\lambda\in(0,\frac{1}{2})$ and
 $\gamma\in(0,1)$ such that
 for any function $u$ that satisfies the following assumptions:
\begin{enumerate}
\item $u(t,x) \leq 1$ in $(-2,0]\times B_2$,
\item the upper local energy-inequality  \eqref{e:EE-loc+} is satisfied,
\item $u$ takes ``some'' early low values in the sense that
\[
\left|\left\{u <{\textstyle\frac12} \right\}\cap (-2,-1]\times B_1\right| \geq \rho\vert B_1\vert,
\]
\item $u$ takes ``enough'' late high values in the sense that
\[
\left|\left\{u > 1 -\lambda^2/2\right\}\cap (-1,0]\times B_1\right| \geq \delta\vert B_1\vert,
\]
\end{enumerate}
then
\begin{equation} \label{intermediate set}
\left|\left\{1/2 \leq u\leq 1 -\lambda^2/2 \right\}\cap((-2,0]\times B_1)\right| \geq
\gamma \,\vert (-2,0] \times B_1\vert.
\end{equation}
\end{lemma}
\begin{remark}
  In view of the proof, a formula can be given for $\gamma$ as a
  function of $\rho, \delta$ and constants only depending on $N,m$ and
  $\alpha$ (see $C_+,C_1,C_2, C_D$ in the proof below). 
  The admissible values for $\lambda$ form an interval of the form$(0,C\rho\delta^2)$ defined
  precisely by  \eqref{cond on lamda}.
 \end{remark}
\begin{remark}
Subsequently, we will use this result with some $\delta=\overline\delta$
 given by lemma \ref{l:DG1-low}.
 \end{remark}
\begin{proof} 
We will follow closely the proofs given in Section 4 of \cite{CCV} and in Section 9 of \cite{CSV}.
As pointed out in \cite{CSV}, the key point is to collect
a super-linear control of the good extra term
\[
\int_{-2}^0 \iint_{B_1 \times B_1} \left(u(t,x)-(1-{\textstyle\frac\lambda2})\right)_+ \left(G(1-{\textstyle\frac\lambda2}) -G(u(y))\right)_+ \frac{\dx\dy}{|x-y|^{N+\alpha}}  \dt
\lesssim \lambda^{1+\varepsilon}
\]
for $\lambda<1/2$ and with some $\varepsilon>1$. In what follows (as in~\cite{CCV}, \cite{CSV}), we will have $\varepsilon+1=2$.
Once this goal has been achieved, then the subsequent steps are a straightforward adaptation of the end of Section 4 of \cite{CCV}.
For the convenience of the reader, we will sketch how the end of the argument goes.

\bigskip
We define for $\lambda < 1/2$,
\begin{eqnarray*}
c_0 
= \frac{1}{2}, \quad c_1 = 1-\frac\lambda2, \quad c_2 = 1-\frac{\lambda^2}2.
\end{eqnarray*}
We fix $\rho \in (0,1/8)$ and we consider
\begin{align*}
&\Ep(t) = \int_{B_1} (u-c_1)_+^2(t,x) \dx, \\
&\get(t) = \iint_{B_1 \times B_1} (u-c_1)_+(t,x) (G(c_1) -G(u(y)))_+ \dx\dy.
\end{align*}
The proof proceeds in several steps.
During the proof, we will freely use  that, on $B_2$:
\begin{equation}\label{uc1trick}
(u-c_1)_+ \leq  \frac{\lambda}{2} \cdot \un_{\{u(x)>c_1\}}.
\end{equation}

\medskip
\noindent \textbf{Step 1:} Using the energy estimate, we first prove in this step that 
\begin{align} \label{G(t)}
\Ep'(t) \leq C_+ \lambda^2 \quad \text{ and } \quad \int_{T_1}^{T_2} \get(t)\dt \leq C_+ \lambda^2
\end{align}
for all $-2 \leq T_1 < T_2 < 0$.
For any $c_0\in(0,c_1)$, 
we can express our assumption \eqref{e:EE-loc+} about the local energy estimate
using its alternate form~\eqref{betterGET} and obtain:
\begin{align*}
\frac{d}{dt} \int_{B_1} (u-c_1)_+^2(t_2,x) \dx & + (G(c_1)-G(c_0)) \left( \int_{B_1 \times B_1} (u-c_1)_+ (x) \un_{\{u(y) < c_0\}}\dx \dy  \right)  \\ 
&\lesssim  C  \left|\{u > c_1 \}\cap (t_1,t_2)\times B_{3/2} \right|  
\lesssim C  \lambda^2.
\end{align*}
For the last inequality, we followed~\cite{CCV}-\cite{CSV} and used~\eqref{uc1trick}.

\medskip\noindent
\textbf{Step 2:} We construct a set of ``early times'' for which
the energy is ``small''.  More precisely, in order to do this, we consider
\[
\Sigma_0 = \{t \in (-2,0): \; |\{u(t,.) < c_0\}\cap B_1|\geq \rho |B_1|/4 \}
\]
and we prove next that
\begin{align}
\label{measure}
\vert\Sigma_0\cap (-2,-1)\vert \geq \frac{\rho}{2}, \\
\label{small-energy}
\int_{\Sigma_0} \Ep(t) \dt \leq C_1 \frac{\lambda^3}{\rho}.
\end{align}
As far as \eqref{measure} is concerned, we remark that the assumptions of the lemma imply
\begin{align*}
\rho\vert B_1\vert &\leq \int_{-2}^{-1}\vert\{u(t) < c_0\}\cap B_1\vert \dt \\
&\leq \int_{\Sigma_0\cap(-2,-1)}\vert\{u(t) < c_0\}\cap B_1\vert \dt + \int_{(-2,-1)\setminus\Sigma_0}\frac{\rho}{4}\vert B_1\vert \dt \\
&\leq \vert B_1\vert\vert \Sigma_0\cap(-2,-1)\vert + \frac{\rho}{2}\vert B_1\vert.
\end{align*}
In order to get \eqref{small-energy}, we first remark that \eqref{G(t)} yields
\[
C \lambda^2 \geq  (G(c_1)-G(c_0)) \int_{\Sigma_0} \int_{B_1} \int_{\{u(t,y) < c_0\}\cap B_1}  (u-c_1)_+(t,x) \dx\dy\dt.
\]
Now we use
\( G(c_1) - G(c_0) =
\left(1-\frac{1}{2}\lambda\right)^{m-1}-\left(\frac{1}{2}\right)^{m-1}
\geq \left(\frac{5}{6}\right)^{m-1}- \left(\frac{1}{2}\right)^{m-1}
\gtrsim 1 \) and get
\begin{eqnarray*}
C \lambda^2 &\gtrsim & \rho \int_{\Sigma_0}\int_{B_1} (u-c_1)_+(t,x)\dx\dt \\
&\gtrsim &  \frac{\rho}{\lambda}\int_{\Sigma_0}\int_{B_1} (u-c_1)_+^2(t,x) \dx\dt
\end{eqnarray*}
using~\eqref{uc1trick} again.

\medskip\noindent 
\textbf{Step 3:} We now consider the following set of ``early times''
for which the energy is small:
\[
\tilde{\Sigma_0} = \{t \in \Sigma_0\cap(-2,-1) : \;  \Ep (t)\leq \frac{C_2}2 \delta \lambda^2 \}
\]
with $C_2$ to be chosen later, and we prove that it has a positive measure, \textit{i.e.}
\begin{equation} \label{measure1}
\vert\tilde{\Sigma_0}\vert \geq \frac{\rho}{4}
\end{equation}
for $\lambda$ small enough. Let $F$ denote  $\Sigma_0\setminus\tilde{\Sigma_0}$. Using \eqref{small-energy} we can write
\[
\vert F\vert = \int_F\dt \leq \frac{2}{C_2 \lambda^2\delta} \int_F \Ep(t) \dt \le \frac{2 C_1 \lambda}{C_2\delta\rho} \le \frac{\rho}2
\]
as soon as
\[
\lambda \leq \frac{C_2 \delta \rho^2}{4C_1}
\]
and we get \eqref{measure1} from \eqref{measure}.

\medskip\noindent
\textbf{Step 4:} We next construct a set of ``late times'' for which
the energy is ``large''. Precisely, we consider
\begin{align*}
\Sigma_2 = \{t\in (-2,0); \vert\{u(t)>c_2\}\cap B_1\vert\geq\frac{\delta}{4}\vert B_1\vert\}
\end{align*}
and we prove that
\begin{align}
\label{measure2}
\vert\Sigma_2\cap(-1,0)\vert \geq \frac{\delta}{2} \\
\label{large-energy}
\forall t \in \Sigma_2, \qquad \Ep (t) \geq C_2 \delta \lambda^2
\end{align}
for $C_2 = |B_1|/64$. 
Estimate~\eqref{measure2} is obtained as above from the assumption of
the lemma. As far as \eqref{large-energy} is concerned, we write for all
$t\in \Sigma_2$ that
\begin{align*} 
\Ep(t) &= \int_{B_1} (u-c_1)_+^2(t,x) \dx \\
 &\geq (c_2-c_1)^2 |B_1 \cap \{\{u(t,x) > c_2\}| \\
&\geq \frac{\lambda^2 (1-\lambda)^2\delta |B_1|}{16}  \\
&\geq \frac{|B_1|}{64} \delta \lambda^2
\end{align*}
for $\lambda \le \frac12$.

\bigskip
\noindent 
\textbf{Step 5:} In this step, we prove that the energy $\Ep$ takes
intermediate values between $C_2 \delta \lambda^2/2$ and
$C_2 \delta \lambda^2$ ``often enough''. Precisely, we consder
\[
D = \{t\in (-2,0); \frac{C_2}2 \lambda^2\delta \leq \Ep(t) \leq C_2 \lambda^2\delta\}
\]
and we prove that 
\begin{align}
\label{measureD}
\vert D\vert \geq \delta C_D \\
 \label{measure3}
\vert D\setminus \Sigma_0 \vert \geq \frac{\vert D\vert}{2}
\end{align}
with $C_D=|C_2|/(2C_+)$.  We start with
\eqref{measureD} by picking a time $T_0 \in \tilde{\Sigma}_0$ where
$\tilde{\Sigma}_0$ (it has a positive measure thanks to \eqref{measure1}) and
$T_2 \in \Sigma_2 \cap (-1,0)$ (it has a positive measure thanks to
\eqref{measure2}).  Consider the truncature function
\[ \mathcal{T} (r) = \max (\min (r,C_2 \delta \lambda^2),C_2 \delta \lambda^2 /2).\] 
Remark that $\mathcal{T}'(r) = \un_{\{ C_2 \delta \lambda^2 /2 \le r \le C_2 \delta \lambda^2 \}}$. 
Then 
\begin{align*}
 \frac{C_2}2 \delta \lambda^2 & \le \mathcal{T} \circ E (T_2) - \mathcal{T} \circ E(T_0 ) \\
& \le \int_{T_0}^{T_2} E'(t) \mathcal{T}' (E(t)) \dt \\
& \le C_+ \lambda^2 |D|
\end{align*}
where we used \eqref{e:EE-loc+}.

\smallskip
As far as \eqref{measure3} is concerned, we use the definition of
$\sigma_2$, \eqref{small-energy} and \eqref{measure3} in order to get
\[
\vert D\cap \Sigma_0 \vert\frac{C_2 \lambda^2\delta}{2} \leq \int_{D\cap \Sigma_0}E(t)\dt
\leq  C_1 \frac{\lambda^3}{\rho} \leq  \frac{C_D \delta}{2} \times \frac{C_2 \lambda^2\delta}{2} 
\leq  \frac{\vert D\vert}{2}\frac{C_2 \lambda^2\delta}{2},
\]
as soon as
\begin{align} \label{cond on lamda}
0<\lambda \leq \frac{C_2 C_D \rho \delta^2}{4C_1}\cdotp
\end{align}

\medskip\noindent
\textbf{Step 6:} We will pick up an intermediate set in $D\setminus\tilde{\Sigma_0}$ with a nontrivial measure. Precisely, 
for $t\in (-2,0)\setminus(\tilde{\Sigma_0}\cup\Sigma_2)$, we have (recall $\delta \le \frac12$ and $\rho \le \frac12$)
\begin{align*}
\vert B_1\vert &= \vert\{u(t) < c_0\}\cap B_1\vert + \vert\{u(t) > c_2\}\cap B_1\vert + \vert\{c_0 < u(t) < c_2\}\cap B_1\vert \\
&\leq \frac{\rho}{2}\vert B_1\vert + \frac{\delta}{2}\vert B_1\vert + \vert\{c_0 < u(t) < c_2\}\cap B_1\vert \\
&\leq \frac{1}{2}\vert B_1\vert + \vert\{c_0 < u(t) < c_2\}\cap B_1\vert.
\end{align*}
Hence for all $t\in (-2,0)\setminus(\tilde{\Sigma_0}\cup\Sigma_2)$ we
have
\begin{align*}
\vert\{c_0 < u(t) < c_2\}\cap B_1\vert \geq \frac{1}{2}\vert B_1\vert.
\end{align*}
Moreover $D\setminus\tilde{\Sigma_0}\subset (-2,0)\setminus(\tilde{\Sigma_0}\cup\Sigma_2)$. So we conclude
\begin{align*}
\int_{-2}^0 \vert\{c_0 < u(t) < c_2\}\cap B_1\vert \dt &\geq \int_{(-2,0)\setminus(\tilde{\Sigma_0}\cup\Sigma_2)} \vert\{c_0 < u(t) < c_2\}\cap B_1\vert \dt \\
&\geq \frac{1}{2}\vert B_1\vert \vert(-2,0)\setminus(\tilde{\Sigma_0}\cup\Sigma_2)\vert \\
&\geq \frac{1}{2}\vert B_1\vert \vert D\setminus\tilde{\Sigma_0}\vert \\
&\geq \frac{1}{2}\vert B_1\vert.\frac{\vert D\vert}{2} \\
&\geq \frac{C_D \delta }{4} \vert B_1\vert.
\end{align*}
Hence the lemma is proved with $\gamma = \frac{C_D \delta}{4}$.
\end{proof}

\section{Lowering the maximum, improved}

We are now in a position to prove the improved oscillation reduction
result from above. We follow the argument given in Section 10 of
\cite{CSV}. The key will be a proper rescaling of the solution.
\begin{lemma}[Lowering the maximum, improved]
\label{l:DG1-low-improved} Let $\alpha \in (0,2)$.
We take $\epsilon_0$ from theorem~\ref{thm:EE}.
For any $\mu\in(0,1/2]$ and $\rho\in(0,1)$, there exists $\mu^* \in (0,1)$
such that for any function $u$ that satisfies the following assumptions:
\begin{enumerate}
\item $u$ satisfies
\[
 u(t,x) \leq 1 \quad \text{ in } (-2,0]\times B_2
\]
\item the upper local energy-inequality  \eqref{e:EE-loc+} is satisfied,
\item $u$ takes ``some'' early low values in the sense that
\[
\left|\{u < \mu \}\cap (-2,-1]\times B_2 \right| \ge  \rho \, |B_2|,
\]
\end{enumerate}
then \( u(t,x) \leq \mu^* \) in $(-1/2,0]\times B_{1/2}$. 
Note that  the value of $\mu^\ast$ depends only on the dimension $N$, on $\gamma, \lambda$ from lemma~\ref{lvi}, on $\rho, \mu$ and $\epsilon_0$.
\end{lemma}
\begin{remark}
Note that the major difference with the first De Giorgi lemma~\ref{l:DG1-low} is that
the value of $\rho$ is now arbitrary while, previously, it was fixed to $\rho=1-\overline\delta$.
Also note that now, as we apply the intermediate values lemma~\ref{lvi}, the full length of~\eqref{e:EE-loc+} is required, \textit{i.e.} the ``good extra term'' plays a crucial role.
Lastly, there is a time-gap (from $t=-1$ to $t=-1/2$) between the third assumption
and the conclusion.
\end{remark}


\begin{proof} The key of the proof consists in applying  lemma~\ref{lvi} to a sequence
of functions until all the space-time available for the intermediary values is spent.
From then on, we will know that $u$ is mostly low-valued on the ``late'' times, \textit{i.e.} on $(-1,0]\times B_1$.
The first De Giorgi lemma \ref{l:DG1-low} will then be applied with a high threshold
and will reduce the maximum, but only on ``late'' times compared to its domain of application.
This step is thus responsible for a small but necessary time-gap between the assumptions
and the conclusion and we can only improve the maximum on $(-1/2,0]\times B_{1/2}$.
The first step consists in checking the assumptions of
lemma~\ref{lvi} on a sequence of ``pushed down and rescaled'' versions of $u$.

  \paragraph{Choice of constants.}
First, we take the values of $\lambda<1/2$ and $\gamma$ given in lemma \ref{lvi}.
Next, we consider \[j_0=\left\lceil\frac{|(-2,0]\times B_1|}{\gamma}\right\rceil.\]
Finally, we take the value $\overline\delta$ given by lemma~\ref{l:DG1-low}
when applied to $\overline\mu=1-\lambda^{2j_0+2}$. 
  
  \paragraph{Claim 1.}
Our first claim is that the functions defined for $1\leq j\leq j_0$ by
\[
u_j(t,x)=\frac{ u(t,x) -(1-\lambda^{2})(1+\lambda^2+\lambda^4+\ldots+\lambda^{2j-2}) }{ \lambda^{2j} }  = \frac{u(t,x)-(1-\lambda^{2j})}{ \lambda^{2j} }
\]
satisfy the local energy estimates \eqref{e:EE-loc+} with uniform constants.
Let us observe that as $j\to\infty$, one has $\lambda^{2j}u_j(t,x) \to u(t,x)-1\leq 0$ on $B_{2^{1/\epsilon_0}}$ so that $u_j$ may take some negative values.
Equivalently, the sequence is defined iteratively by
\begin{eqnarray*}
u_{j+1}(t,x)=\frac{1}{\lambda^2}\left(u_j(t,x)-(1-\lambda^2)\right),
\end{eqnarray*}
starting from $u_1(t,x)=\lambda^{-2}(u(t,x)-(1-\lambda^2))$.

\medskip
For any $c_j>0$, let us repeatedly apply our assumption~\eqref{e:EE-loc+} to the function $u$, with the cut-off constant
\[
c= \lambda^{2j} c_j + (1-\lambda^{2j}) >1/4,
\]
radii $0<r<R<2^{1/\epsilon_0}$ and start and stop times $-2<t_1 < t_2 <0$.
Using \eqref{betterGET} to express the good extra term, we get:
\begin{align*} 
\nonumber
\int_{B_r} & (u-(1-\lambda^{2j})-\lambda^{2j} c_j)_+^2(t_2,x) \dx 
+ \int_{t_1}^{t_2} \left(\int_{B_{r}}(u-(1-\lambda^{2j})-\lambda^{2j} c_j)_+^p(x) \dx\right)^\frac{2}{p} \dt \\
&+ \int_{t_1}^{t_2} \left( \iint_{B_r \times B_r} 
\left(u(x)-(1-\lambda^{2j})-\lambda^{2j} c_j\right)_+
\left(G ((1-\lambda^{2j})+\lambda^{2j} c_j) -G(u(y))\right)_+ 
\dx \dy  \right) \dt 
\nonumber  \\ 
&\lesssim \int_{B_R} (u(t_1,x)-(1-\lambda^{2j})-\lambda^{2j} c_j)_+^2 \dx + C (R-r)^{-2} \left|\{u - (1-\lambda^{2j}) 
> \lambda^{2j} c_j \}\cap (t_1,t_2)\times B_{R} \right|.
\end{align*}
We deduce from the previous inequality that $u_j-c_j=(u-c)/\lambda^{2j}$
satisfies the following local energy estimate
\begin{align*} 
\nonumber
\int_{B_r} & (u_j(t_2,x)- c_j)_+^2 \dx 
+ \int_{t_1}^{t_2} \left(\int_{B_{r}}(u_j(t,x)- c_j)_+^p \dx\right)^\frac{2}{p} \dt \\
&+ \int_{t_1}^{t_2} \left( \iint_{B_r \times B_r} (u_j- c_j)_+ (x) 
(G ((1-\lambda^{2j})+\lambda^{2j} c_j) -G(u(y))) \dx \dy  \right) \dt 
\nonumber  \\ 
&\lesssim \int_{B_R} (u_j(t_1,x)- c_j)_+^2 \dx + C (R-r)^{-2} \lambda^{-4j} \left|\{u_j > c_j \}\cap (t_1,t_2)\times B_{R} \right|.
\end{align*}
As $\lambda^{-1}>1$, we have $\lambda^{-4j}\leq \lambda^{-4j_0}$ if $j\leq j_0$.
We conclude that,  as long as $1\leq j\leq j_0$,
all the functions $u_j$ satisfy the local energy estimates \eqref{e:EE-loc+}
with uniform constants. Moreover, $c_j>0$ can be arbitrary.

\paragraph{Claim 2.} 
We also claim that the early low-values assumption of lemma~\ref{lvi} does hold for $u_j$.
Indeed, as $\lambda<1$ then for any $\mu<1$, the inequality $u_j(t,x) < \mu$ implies $u_{j+1}(t,x) < \mu$ hence
\[
|\{u_j < 1/2 \}\cap (-2,-1]\times B_1| \ge
|\{u_j < \mu \}\cap (-2,-1]\times B_1| \ge
 |\{ u < \mu \}\cap (-2,-1]\times B_1| \ge \rho|B_1|.
\]
As $\mu\leq 1/2$, the early low-values assumption of lemma~\ref{lvi} is satisfied.

\paragraph{Main Step.}  Let us now reason by contradiction.
We assume that for any $j\in[1,j_0]$ one has
\[
|\{ u_j > 1-\lambda^2/2\} \cap (-1,0] \times B_1 | \ge
\overline{\delta} |B_1|.
\]
Then the lemma~\ref{lvi} on intermediate values can be applied to $u_j$ and implies that
\[ |\{ 1/2 \le u_j \le 1-\lambda^2 /2  \} \cap (-2,0] \times B_1 | \ge \gamma |(-2,0] \times B_1|.\]
Translating this for the function $u$, we get 
\[ \left|\left\{ 1 - \frac{\lambda^{2j}}2 \le u \le 1-\frac{\lambda^{2j}}2
    -\frac{\lambda^{2j+1}}2 \right\} \cap (-2,0] \times B_1 \right|
    \ge \gamma \,  |(-2,0] \times B_1|.\]
This implies in particular 
\[ \left|\left\{ 1 - \frac{\lambda^{2j}}2 \le u \le 1 -\frac{\lambda^{2j+2}}2 \right\}
\cap (-2,0] \times B_1 \right|
\ge \gamma \, |(-2,0] \times B_1|.\]
But these intermediate level sets are disjoint and of positive measure so there can be
only at most $j_{0}-1$ of them in the space-time ball $(-2,0] \times B_1$.
The original assumption is false.

\medskip
In particular, there exists $j_1 \le j_0$ such that 
\[
|\{u_{j_1}>1 -\lambda^2/2\}\cap (-1,0]\times B_1| < \overline{\delta} |B_1|.
\]
As $\overline\mu=1-\lambda^{2j_0+2}$, this translates back to $u$ as
\begin{equation}\label{mlw}
|\{u> {\textstyle\frac{1+\overline\mu}{2}}\}\cap (-1,0]\times B_1|
\leq |\{u>1 -\lambda^{2j_1+2}/2\}\cap (-1,0]\times B_1| < \overline{\delta} |B_1|.
\end{equation}
We want to apply the first De Giorgi lemma \ref{l:DG1-low}
to $u$  with $\overline\mu=1-\lambda^{2j_0+2}$.
However, \eqref{mlw} only states that $u$ is ``mostly low valued'' at late times
while lemma \ref{l:DG1-low} requires $u$ to be ``mostly low valued'' for all times.

\medskip
We thus consider $\tilde{u}(t,x)=u(t/2^\alpha,x/2)$,
which satisfies
\[
|\{\tilde{u}> {\textstyle\frac{1+\overline\mu}{2}}\}\cap (-2,0]\times B_2| < \overline{\delta} |B_1|.
\]
because $2^\alpha>2$ (note that we use here again that $\alpha\geq1$).
Applying  lemma \ref{l:DG1-low} to $\tilde{u}$, we get:
\[
\tilde{u}(t,x) \leq \frac{3+\overline{\mu}}{4}=\mu^\ast \quad \text{ on } (-1,0]\times B_1
\]
with $\mu^* = 1- \lambda^{2(j_0+1)}/4$.
Hence $u(t,x) \leq \mu^*$ on $(-1/2,0]\times B_{1/2}$.
\end{proof}

\section{Proof of the main theorem}\label{mainproof}

In this section, we alternatively use 
the lemma of De Giorgi on increasing the infimum (lemma~\ref{l:DG-pull})
and the improved lemma about lowering the maximum (lemma~\ref{l:DG1-low-improved})
in order to prove theorem~\ref{thm:main}.
\begin{proof}[Proof of theorem~\ref{thm:main}]

\bigskip
  We now consider a weak solution of \eqref{e:main}-\eqref{e:ic}
 associated with an initial data
 \[ u_0\in L^1\cap L^\infty(\RN;\mathbb{R}_+).\] 
We know from \cite[{theorem 2.6}]{BIK}, which we recalled here as theorem~\ref{existence},
  that this solution is globally bounded in
  $[0,+\infty) \times \R^N$, by a constant $M$ that depends on
  $\|u_0\|_{L^1 (\R^N)}, \|u_0\|_{L^\infty (\R^N)}$. 
  To prove theorem~\ref{thm:main}, we want to study its H\"older regularity on
  some interval $[T_0,T_1]$ with $0<T_0<T_1$.
  
  \medskip
  When $\alpha\in(0,1]$, one replaces $u$ by the properly drifted $\bar{u}(t,x)$ defined by~\eqref{uDrift1}-\eqref{uDrift2} so that
  the upper local energy-inequality  \eqref{e:EE-loc+} holds. The scaling transforms that appear in the rest of the proof can be
  adapted to the drift term, as in \cite[\S12.3]{CSV}.
  
 \medskip
  We can translate the time interval
    and study the equation in $(-2,0] \times \R^N$. It is then sufficient to prove that it
    is H\"older continuous at the point~$(t_0,x_0)=(0,0)$. Using the scale invariance,
    we can assume without loss of generality that~$M =1$
    (by choosing $A=1/M$ and $C=M^{\frac{m-1}\alpha}$ in \eqref{scaling}).  In particular,
  \begin{equation}\label{mildgrowth}
  0\leq u(t,x) \le 1 + \Psi_{\eps_0} (x) \quad \text{ for } (t,x) \in
  (-2,0]\times \R^N.
  \end{equation}
where $\Psi_{\eps_0}$ is the mildly-growing function defined by~\eqref{defpsiepsilon}.

\medskip
 In order to apply lemma~\ref{l:holder},
 we are going to prove that the oscillation of $u$ around the point
  $(t_0,x_0)=(0,0)$ decays algebraically on $(-r,0]\times B_r$
  as $r\to0$.
 More precisely, we will show subsequently that if the solution
 $u$ satisfies~\eqref{mildgrowth}, then
 $\underset{(-1,0]\times B_1}{\operatorname{osc}u} \leq \omega_0<1$.
 Thanks to \eqref{scaling} we can then construct a sequence of rescaled solutions
 \[
 v_{n+1}(t,x) = \frac{ v_n(t/\tau_n, x/\kappa_n) }{ (\tau_n/\kappa_n^\alpha)^{1/(m-1)}} \qquad\text{with}\qquad
 v_1(t,x)=u(t,x)
 \]
and scaling parameters $\tau_n,\kappa_n\geq2$ such that $\tau_n/\kappa_n^\alpha \leq 1$. One can adjust
the parameters such that all the $v_{n}$ satisfy \eqref{mildgrowth}.
Note that the values of the pair $(\tau_n,\kappa_n)$ can alternate between a few
universal choices from one iteration to the next, but overall, it has no detrimental effect.

Iterating this construct gives
a dyadic formulation of the assumption of lemma~\ref{l:holder}, which can then
ultimately be applied to $u$ and provides the desired H\"older regularity.
  
\medskip
Let us now explain the fine details of the process that reduces the oscillation of $v_{n}$ on $(-1,0]\times B_1$. We consider an increasing sequence of thresholds
\[
\mu_n = 1- \frac{1}{2^n}\cdotp
\]
We take $\delta_n$ to be the value of $\underline{\delta}$ associated
 with $\underline{\smash\mu}=\mu_n$ by  lemma~\ref{l:DG-pull}.
 We will successively distinguish two mutually exclusive cases. 
\nopagebreak

\medskip\noindent
$\bullet$ The first possibility is that
\begin{equation}\label{firstoption}
| \{v_n \geq (1+2\mu_n)/3\} \cap (-2,-1]\times B_2 | \geq
(1-\delta_n)|(-2,0) \times B_2|.
\end{equation}
In particular, one has
\[
| \{v_n \geq (1+2\mu_n)/3\} \cap (-2,0]\times B_2 | \geq
(1-\delta_n)|(-2,0) \times B_2|.
\]
In this case, we can apply lemma~\ref{l:DG-pull} with $\underline{\smash\mu}=\mu_n$
 and get that $u$ satisfies
\[ v_n(t,x) \geq \mu_n \text{ in } (-1,0] \times B_1.\]
The oscillation of $v_n$ has thus decreased from $1$ on $(-2,0]\times B_2$
to $\underset{(-1,0]\times B_1}{\operatorname{osc}v_n} \leq 1-\mu_n = 2^{-n}$.

\medskip
For the subsequent rescaling, we take
\[
v_{n+1}(t,x)= v_n (t/\tau,x/\kappa) \qquad\text{with}\qquad 
\kappa=2 \quad\text{and}\quad
\tau=\kappa^\alpha\geq 2
\]
which, according to~\eqref{scaling}, is also a solution of \eqref{e:main}.
Moreover, it satisfies
\begin{equation}
0\leq (1-2^{-n})\,\un_{B_2}(x)
\leq v_{n+1}(t,x) \leq 1+\Psi_{\eps_0}(x) \qquad\text{on}\qquad (-2,0]\times\RN
\end{equation}
so in particular~\eqref{mildgrowth} holds again for $v_{n+1}$.

\medskip\noindent
$\bullet$ In case \eqref{firstoption} fails, the alternative reads
\begin{equation}\label{secondoption}
 |\{v_n \geq (1+2\mu_n)/3\} \cap (-2,-1]\times B_2 | <
(1-\delta_n)|(-2,0] \times B_2|
\end{equation}
which implies
\[
 |\{v_n < (1+2\mu_n)/3\} \cap (-2,-1]\times B_2 | \geq
2 \delta_n |B_2|.
\]
In this case, we can apply lemma~\ref{l:DG1-low-improved} with $\mu=(1+2\mu_n)/3$ and $\rho=2\delta_n$
and get that 
\[ v_n(t,x) \le \mu^* \text{ in } (-1/2,0] \times B_{1/2}.\]
The oscillation of $v_n$ has thus decreased
to $\underset{(-1/2,0]\times B_{1/2}}{\operatorname{osc}v_n} \leq \mu^\ast$.

\medskip
We then consider the function
\[ v_{n+1} (t,x) = \frac{v_n (t/\tau, x/\kappa)}{\mu^*}\]
with $\tau = (\mu^*)^{m-1} \kappa^\alpha$.
Note that $\tau/\kappa^\alpha<1$.
Thanks to \eqref{scaling}, we know that $v_{n+1}$ is
still a weak solution of \eqref{e:main} and that
\[\begin{cases}
 v \le 1 & \text{ in } (-\tau/2,0] \times B_{\kappa/2} , \\
 v \le \frac{(|x/\kappa|^{\eps_0}-2)_++1}{\mu^*} & \text{ in } (-\tau,0] \times \R^N.
\end{cases}\]
It is not difficult to check that for $\kappa \ge 2(1 + (\mu^*)^{-1})^{1/\eps}>4$,
then
\[
\frac{(|x/\kappa|^{\eps_0}-2)_++1}{\mu^*}  \leq 1 + \Psi_{\eps_0}(x) \qquad\text{ouside }B_{\kappa/2}.
\]
By also choosing $\kappa \ge (4(\mu^*)^{-(m-1)})^{1/\alpha}$, we get $\tau \ge 4$ and therefore
\[ 0\leq v_{n+1}(t,x) \le 1 + \Psi_{\eps_0} (x) \qquad\text{on}\qquad (-2,0] \times \R^N.\]
This concludes the treatment of the alternative case~\eqref{secondoption}.

\bigskip
Conclusion. In both cases \eqref{firstoption}-\eqref{secondoption},
we have reduced the oscillation of $u$ by at least a
universal factor \[\omega_0= \frac12 \wedge \mu^*<1\] and proposed
a universal rescaling process that brings us back to the initial situation \eqref{mildgrowth}.
As explained in the introduction of this proof,
the oscillation then decays algebraically when zooming in and
this fact achieves the proof of the main theorem.
\end{proof}

\appendix
\section{Useful inequalities}

\begin{lemma}
The following inequalities are valid for any $m>1$
\begin{equation}\label{SV}
\forall a,b\geq0,\qquad
(a-b)(a^{m-1}-b^{m-1}) \geq {\textstyle \frac{4(m-1)}{m^2}} \left(a^{\frac{m}{2}}-b^{\frac{m}{2}}\right)^2
\end{equation}
\begin{equation}\label{SV2}
\forall a,b\geq0,\qquad
(a^2-b^2)(a^{m-1}-b^{m-1}) \geq {\textstyle \frac{8(m-1)}{(m+1)^2}} \left(a^\frac{m+1}{2}-b^\frac{m+1}{2}\right)^2
\end{equation}
and
\begin{equation}\label{VF}
\forall a,b\geq c>0,\qquad
(a-b)^2 \leq c^{-(m-1)}(a^{\frac{m+1}{2}} - b^{\frac{m+1}{2}})^2.
\end{equation}
\end{lemma}

\begin{remark}
The following proof also shows that converse inequalities to \eqref{SV}-\eqref{SV2} are also true:
\begin{gather}
(a-b)(a^{m-1}-b^{m-1}) \leq \left(a^{\frac{m}{2}}-b^{\frac{m}{2}}\right)^2\\
(a^2-b^2)(a^{m-1}-b^{m-1}) \leq \left(a^\frac{m+1}{2}-b^\frac{m+1}{2}\right)^2
\end{gather}
for any $a,b\geq0$.
\end{remark}
\begin{proof}
When $a=0$, all inequalities are obvious, at least once we observe that
\[
m^2-4(m-1)=(m-2)^2\geq0 \qquad\text{and}\qquad
(m+1)^2-8(m-1)=(m-3)^2\geq0.
\]
Again, they are also true when $a=b$.
We can thus assume that $a\neq0$ and $a\neq b$ and consider $\theta=b/a\in \mathbb{R}_+$
but with $\theta\neq1$. We claim that the functions
\[
f(\theta)=\frac{(1-\theta^{\frac{m}{2}})^2}{(1-\theta)(1-\theta^{m-1})},\qquad
g(\theta)=\frac{(1-\theta^{\frac{m+1}{2}})^2}{(1-\theta^2)(1-\theta^{m-1})} 
\qquad\text{and}\qquad
h(\theta)=\frac{1-\theta}{1-\theta^{\frac{m+1}{2}}}
\]
are continuous through $\theta=1$ and
satisfy $\|f\|_{L^\infty(\mathbb{R}_+)}\leq \frac{m^2}{4(m-1)}$, $\|g\|_{L^\infty(\mathbb{R}_+)}\leq\frac{(m+1)^2}{8(m-1)}$
and $\|h\|_{L^\infty(\mathbb{R}_+)}\leq1$.
The inequalities~\eqref{SV}-\eqref{SV2} then follow from $(1-\theta)(1-\theta^{m-1})>0$ and
$(1-\theta^2)(1-\theta^{m-1})>0$
while~\eqref{VF} comes from
\[
(a-b)^2 = a^2 g(\theta)^2 (1-\theta^{\frac{m+1}{2}})^2 \leq
a^{-(m-1)} (a^{\frac{m+1}{2}} - b^{\frac{m+1}{2}})^2 \|g\|_{L^\infty}^2
\]
and the final restriction $a\geq c$.

\medskip
To back up our claim, let us briefly study the functions $f$, $g$ and $h$.
The continuity around $\theta=1$ comes from a simple Taylor expansion:
\begin{align*}
f(\theta) &= \frac{m^2}{4(m-1)} - \frac{m^2(m-2)^2}{192(m-1)}(\theta-1)^2+ O[(\theta-1)^3]\\
g(\theta) &= \frac{(m+1)^2}{8(m-1)} -\frac{(m+1)^2(m-3)^2}{384(m-1)}(\theta-1)^2 + O[(\theta-1)^3]\\
h(\theta)&=\frac{2}{m+1}-\frac{m-1}{2(m+1)}(\theta-1)+O[(\theta-1)^2].
\end{align*}

Moreover, one can check that
$f(\theta)$ and $g(\theta)$ reach a global maximum when $\theta=1$
while $h(\theta)$ is maximal at $\theta=0$, \textit{i.e.}:
\[\forall \theta>0, \qquad
\begin{cases}
1\leq f(\theta) \leq f(1),\\
1\leq g(\theta) \leq g(1),\\
0\leq h(\theta)\leq h(0).
\end{cases}
\]
The lower values follow from the limits at $0$ and $+\infty$, once we know the variations of $f,g,h$.

For example, for any $\theta>0$, one has
\[
h'(\theta)= -\frac{(m-1)\theta^{\frac{m+1}{2}} - (m+1)\theta^{\frac{m-1}{2}} + 2}{2\theta\left(1-\theta^{\frac{m+1}{2}}\right)^2}\leq0
\]
because $(m-1)\theta^{\frac{m+1}{2}} - (m+1)\theta^{\frac{m-1}{2}} + 2$ is a positive function
that vanishes only for $\theta=1$. Indeed, we can rewrite it as a balance of two signed terms
\[
(m-1)\theta^{\frac{m+1}{2}} - (m+1)\theta^{\frac{m-1}{2}} + 2
= \theta^{\frac{m-1}{2}} (m-1)(\theta-1) + 2 (1-\theta^{\frac{m-1}{2}})
\]
whose derivative is
\[
\frac{d}{d\theta}\left[ \theta^{\frac{m-1}{2}} (m-1)(\theta-1) + 2 (1-\theta^{\frac{m-1}{2}}) \right]
= \theta^{\frac{m-3}{2}}(m-1)(\theta-1)\left(\theta+\frac{m+1}{2}\right)
\]
and has therefore the same sign as $\theta-1$.

Similarly, one has $f'(\theta)=\frac{(1-\theta^{m/2})(1-\theta^{(m-2)/2})\mathfrak{F}(\theta)}{(1-\theta)^2(1-\theta^{m-1})^2}$
with $\mathfrak{F}(\theta)= 1-\theta^{m-1}+(m-1)\theta^{m/2}(1-\theta^{-1})$, which (based on
a quick study of $\mathcal{F}'$) has the sign of $(m-1)(2-m)(\theta-1)$.
Therefore $f'(\theta)$ has the same sign as
\[(m-1)(2-m)(\theta-1)(1-\theta^{m/2})(1-\theta^{(m-2)/2})\]
which, for $m>1$ is positive on $(0,1)$ and negative on $(1,+\infty)$.

In the same spirit, one gets $g'(\theta)=\frac{(1-\theta^{\frac{m+1}{2}})\mathfrak{G}(\theta)}{\theta^2(1-\theta^2)^2(1-\theta^{m-1})^2}$
for some auxiliary function $\mathfrak{G}(\theta)\geq0$ and thus $g'(\theta)$ has the same sign as $1-\theta^{\frac{m+1}{2}}$.
\end{proof}

\bibliographystyle{siam}
\bibliography{regularitybiblio}

\end{document}